\documentclass[12pt]{amsart}
\usepackage{etex} 
\usepackage[active]{srcltx}

\usepackage{calc,amssymb,amsthm,amsmath,amscd, eucal,ulem, mathtools}
\usepackage{phaistos}
\usepackage{alltt}
\usepackage{microtype}
\RequirePackage[dvipsnames,usenames]{color}

\input{mabliautoref.sty}
\input{xy}
\xyoption{all}
\input{kmacros3.sty}
\usepackage{tikz}
\usepackage{tikz-cd}
\usepackage[marvosym]{tikzsymbols}
\usepackage{amsfonts, mathrsfs}
\usepackage[cal=boondox, calscaled=1.05]{mathalfa}
\usepackage{calligra}
\usepackage{stmaryrd} 
\usepackage[left=1in,top=1in,right=1in,bottom=1in]{geometry}
\usepackage{bm}
\usepackage{verbatim}
\usepackage{upgreek}

\numberwithin{equation}{theorem}

\renewcommand{\:}{\colon}
\newcommand{\eg}{{\itshape e.g.} }
\renewcommand{\m}{\mathfrak{m}}
\renewcommand{\n}{\mathfrak{n}}

\newcommand{\kay}{\mathcal{k}}

\DeclareMathOperator{\pd}{pd}

\DeclareMathOperator{\Cl}{Cl}

\usepackage{setspace}
\usepackage{hyperref}
\usepackage{enumerate}
\usepackage{graphicx}
\usepackage[all,cmtip]{xy}

\newcommand{\mytau}{{\uptau}}

\theoremstyle{theorem}
\newtheorem{Theoremx}{Theorem}

\renewcommand{\O}{\mathcal O}

\renewcommand{\sH}{\mathcal{H}}

\usepackage{marvosym}

\newcommand{\BCMReg}[1]{{$\textnormal{BCM}_{#1}$-regular}}
\newcommand{\BCMRat}[1]{{$\textnormal{BCM}_{#1}$-rational}}

\begin{document}
\title[]{Covers of rational double points in mixed characteristic}
\author[J.~Carvajal-Rojas]{Javier Carvajal-Rojas}
\address{\'Ecole Polytechnique F\'ed\'erale de Lausanne\\ SB MATH CAG\\MA C3 615 (B\^atiment MA)\\ Station 8 \\CH-1015 Lausanne\\Switzerland \newline\indent
Universidad de Costa Rica\\ Escuela de Matem\'atica\\ San Jos\'e 11501\\ Costa Rica}
\email{\href{mailto:javier.carvajalrojas@epfl.ch}{javier.carvajalrojas@epfl.ch}}
\author[L.~Ma]{Linquan Ma}
\address{Department of Mathematics\\ Purdue University\\  West Lafayette\\ IN 47907\\USA}
\email{\href{mailto:ma326@purdue.edu}{ma326@purdue.edu}}
\author[T.~Polstra]{Thomas Polstra}
\address{Department of Mathematics\\ University of Virginia\\ Charlottesville\\ VA 22903\\USA}
\email{\href{mailto:tp2tt@virginia.edu}{tp2tt@virginia.edu}}
\author[K.~Schwede]{Karl Schwede}
\address{Department of Mathematics\\ University of Utah\\ Salt Lake City\\ UT 84112\\USA}
\email{\href{mailto:schwede@math.utah.edu}{schwede@math.utah.edu}}
\author[K.~Tucker]{Kevin Tucker}
\address{Department of Mathematics\\ University of Illinois at Chicago\\Chicago\\  IL 60607}
\email{\href{mailto:kftucker@uic.edu}{kftucker@uic.edu}}

\thanks{JCR was supported in part by NSF CAREER Grant DMS \#1252860/1501102 and by the ERC-STG \#804334, LM was supported in part by NSF Grant \#1901672, and was supported in part by \#1836867/1600198 when preparing this article, TP was supported in part by NSF Grant \#1703856 and NSF Grant \#2101890, KS was supported in part by NSF CAREER Grant \#1252860/1501102 and NSF grant \#1801849, KT was supported in part by  NSF grant DMS \#1602070
and \#1707661, and by a fellowship from the Sloan Foundation.}



\begin{abstract} We further the classification of rational surface singularities. Suppose $(S,\fran, \kay)$ is a $3$-dimensional strictly Henselian regular local ring of mixed characteristic $(0,p>5)$. We classify functions $f$ for which $S/(f)$ has an isolated rational singularity at the maximal ideal $\fran$. The classification of such functions are used to show that if $(R,\fram, \kay)$ is an excellent, strictly Henselian, Gorenstein rational singularity of dimension 2 and mixed characteristic $(0,p>5)$, then there exists a split finite cover of $\Spec(R)$ by a regular scheme. We give an application of our result to the study of 2-dimensional \BCMReg{} singularities in mixed characteristic.
\end{abstract}

\maketitle

\section{Introduction}

The study of surface singularities in algebraic geometry is a classical subject. Of particular interest is the collection of normal surfaces that remain normal under blowups of singular points. Such surfaces are seen to be cohomologically trivial; if $X$ belongs to the set of surfaces just described and $X' \xrightarrow{\rho} X$ is proper and birational, then $R^1\rho_*\O_{X'}=0$. Such surfaces are said to have \emph{rational singularities}, their study was initiated by Du Val in \cite{DuVal1, DuVal2, DuVal3}, and defined by Artin in \cite{ArtinOnisolatedrational}.

Suppose that $X$ is the spectrum of a local $2$-dimensional ring $(R,\fram,\kay)$ with a Gorenstein isolated rational singularity at the closed point. It is known that, under these assumptions, we have $\widehat{R}\cong S/(f)$ where $(S,\fran,\kay)$ is a regular local ring of dimension $3$ and $f\in \fran^2-\fran^3$ (see \autoref{lem.BS}).

Following tradition, such hypersurface singularities are referred to as \emph{rational double points}.  Suppose further that $S$ is strictly Henselian.    That is,  $S$ satisfies Hensel's lemma (e.g.  $S$ is complete) and has separably closed residue field; see \cite[I, \S4]{MilneEtaleCohomology} for further details. If $S$ contains a field then it is known that there exists a finite cover $Y \xrightarrow{\pi} X$ such that $Y$ is the spectrum of a regular local ring.  It is known that, in the equicharacteristic $0$ and prime characteristic $p>5$ scenarios, the induced map $\O_X \to \pi_* \O_Y$ is split as a map of $\O_X$-modules. Before discussing details and appropriate references of the equicharacteristic results, we state our main theorem which generalizes this result to the mixed characteristic setting.  The proof of this is completed in \autoref{sec.FiniteCoversOfRationalDoublePoints}.

\begin{Theoremx}\label{Main Theorem covers by regular ring}
Let $(R,\fram,\kay)$ be an excellent and strictly Henselian local ring of mixed characteristic $(0,p>5)$.\footnote{For instance,  it could be the strict Henselization of an excellent mixed characteristic local ring or well its completion in case its residue field is separably closed. Strict Henselizations can be thought of as the geometric germs with respect to the \'etale topology.  See \cite[I, \S4]{MilneEtaleCohomology}.} Suppose $R$ is a Gorenstein rational singularity of dimension $2$. Then $R$ is a rational double point and there exists a finite cover $Y\xrightarrow{\pi}X=\Spec(R)$ such that $\O_X \to \pi_* \O_Y$ splits as a map of $\O_X$-modules and $Y$ is a regular scheme.
\end{Theoremx}

In the scenario that $R$ contains a field of characteristic $0$, the existence of a finite cover of $X$ by a regular scheme is accomplished by realizing $R$ as a quotient singularity of a finite subgroup of $G\subset \mbox{SL}_2$, see \cite{PrillLocalclassification}. Every equicharacteristic $0$ normal domain is a {\it splinter}, i.e., it splits off from all its module-finite extension. In particular, the finite cover splits.

If $R$ contains a field of prime characteristic $p>0$, Artin provides an explicit description of all possible functions $f\in \fran^2-\fran^3$, in terms of choices of minimal generators of the maximal ideal $\fran$, so that $R=S/(f)$ is a rational double point. The explicit descriptions provided by Artin can then be used to show the existence of a finite cover $Y\xrightarrow{\pi} X$ such that $Y$ is regular, \cite{ArtinCoveringsOfTheRtionalDoublePointsInCharacteristicp} (note Lipman had previously worked out explicit equations in the $\textnormal{E}_8$ case, even in mixed characteristic; see \cite{LipmanRationalSingularities}).  It is then straightforward to use Artin's classification of rational double points in prime characteristic $p>5$ to verify that all such singularities are $F$-regular (a class of singularities coming out of Hochster and Huneke's tight closure theory \cite{HochsterHunekeTightClosureAndStrongFRegularity}). In particular, rational double points of prime characteristic $p>5$ are splinters, \cite[Theorem~5.25]{HochsterHunekeTCParameterIdealsAndSplitting}, and therefore the finite cover by a regular scheme must split.  However, there exists rational double points in characteristic  $2$, $3$, and $5$ which are not $F$-regular.\footnote{For example,  if $S$ is the completion of $\overline{\mathbb{F}}_3[x,y,z]$ at the maximal ideal $(x,y,z)$, then the function $f=x^2+y^3+z^5$ defines a rational double point which is not $F$-regular (in fact, it is easy to see using Fedder's criterion \cite{FedderFPureRat} that it is not even $F$-pure). We will see a similar phenomenon in mixed characteristic: \autoref{example}.} In particular, the finite cover by a regular scheme cannot split since direct summands of regular rings are $F$-regular, \cite[Proposition~4.12]{HochsterHunekeTC1}.  It is worth noting that $F$-regular singularities are precisely the positive equicharacteristic BCM-regular singularities which are the subject of \autoref{sec.CyclicCoversOfBCM} in mixed characteristic.

Similar to the methodology of Artin, to prove Theorem~\autoref{Main Theorem covers by regular ring} we will first classify the functions in the maximal ideal of a $3$-dimensional regular local ring which define a rational double point.  In fact, Lipman already did this for the $\textnormal{E}_8$ case assuming the residual characteristic $p > 5$ as we do; see \cite[Section 25]{LipmanRationalSingularities}.  We were heavily inspired by his work.  We will prove the following in \autoref{sec.Classification}.

\begin{Theoremx}\label{Main Theorem description of functions} Let $(S,\fran,\kay)$ be a $3$-dimensional complete regular local ring of mixed characteristic $(0,p>5)$ with separably closed residue field and $f\in \fran^2-\fran^3$ so that $X=\Spec(R = S/(f))$ is a rational double point. Then,  there exists a choice of minimal generators $x,y,z$ of the maximal ideal of $S$ so that, up to multiplication by a unit,  $f$ can be written in one of the following forms:
\[
\begin{cases}
x^2 +y^2 + z^{n+1} & \textnormal{A}_n \quad (n\geq 1)  \\
x^2+y^2z+z^{n-1} & \textnormal{D}_n \quad (n\geq 4) \\
x^2+y^3+z^{4} & \textnormal{E}_6  \\
x^2+y^3+yz^{3} & \textnormal{E}_7 \\
x^2+y^3+z^{5} & \textnormal{E}_8
\end{cases}
\\
\]
However,  in contrast to the  equicharacteristic case,  one of the above forms for $f$ can yield two non-isomorphic rational double points $S/f$; see \autoref{ex.OnefManyRDPs} below.
\end{Theoremx}
The notations $\textnormal{A}_n$, $\textnormal{D}_n$, $\textnormal{E}_6$, $\textnormal{E}_7$, and $\textnormal{E}_8$ are referred to as the type of the ring $R=S/(f)$ and correspond to the graph of the minimal resolution of $X$ obtained by quadratic transforms as described in \cite{LipmanRationalSingularities,LipmanDesingularizationOf2Dimensional}.  It is well known to experts that any complete Gorenstein rational singularity can be expressed as a hypersurface singularity $S/(f)$ (as we assume in the theorem),  but also see \autoref{lem.BS}.

\begin{example} \label{ex.OnefManyRDPs}
Unlike the equicharacteristic scenario,  the type of the singularity does not determine the singularity up to isomorphism.  Indeed, if $S$ is a strictly Henselian regular local ring of dimension $3$ and of mixed characteristic $(0,p)$,  then two elements $f,g\in S$ can define a rational double point of the same type but the rings $S/(f)$ and $S/(g)$ may not be isomorphic. For example,  let $\kay$ be an algebraically closed field of prime characteristic $p$ and $W(\kay)$ the ring of Witt vectors over $\kay$.  For the sake of concreteness,  one may take $\kay$ to be the algebraic closure of $\mathbb{F}_p$ and so $W(\kay)$ is the ring of integers in the completion of the maximal unramified extension of $\mathbb{Q}_p$.  More on Witt vectors can be found in \cite[Chapter II, Section 6]{SerreLocalFields} or \cite{RabinoffWittVectors}.
Then,  the singularities $W(\kay)\llbracket y,z\rrbracket\big/\bigl(p^2 +y^2 + z^3\bigr)$ and $W(\kay) \llbracket y,z\rrbracket\big/\bigl(z^2 +y^2 + p^3\bigr)$ are both of type $\textnormal{A}_2$, however are not isomorphic.  To see why these are not isomorphic,  observe that any ring homomorphism in between them will send $p$ to $p$.  Therefore if these were isomorphic,  they are isomorphic after moding out by $p$.  In our case,  we would get an isomorphism $\kay\llbracket y,z \rrbracket/(y^2+ z^3) \cong \kay\llbracket z,y \rrbracket/(z^2+ y^2) $ which is absurd (the former is a cusp whereas the latter is not an integral domain).  Of course, analogous examples can be constructed likewise for the other types.
\end{example}

As an application of Theorem~\autoref{Main Theorem covers by regular ring}, we will show in \autoref{sec.CyclicCoversOfBCM} that every $2$-dimensional \BCMReg{} singularity of mixed characteristic $(0,p>5)$ is a finite direct summand of a regular ring.

In summary, this article concerns itself with the classification of $2$-dimensional rational hypersurface singularities $R = S/(f)$ in the mixed characteristic case.  The theory is well understood when $S$ is equicharacteristic and we refer the reader to \cite{ArtinCoveringsOfTheRtionalDoublePointsInCharacteristicp,LipmanRationalSingularities,GreuelKroningSimpleSingularitiesInPosChar} for details.  Classification of rational double points in small mixed characteristics would be desirable and is still open.  We do not attempt it but it would be a natural project for someone to carry out in the future.

\subsection*{Acknowledgements}
The authors thank H.~Esnault, J.~Lipman, B.~Martinova, and C.~Xu for valuable conversations.  Related study and discussion of surface singularities also took place at two AIM SQUARES in 2017 and 2019 attended by the authors Ma, Schwede and Tucker (and also attended by C.~Liedtke, Z.~Patakfalvi, J.~Waldron, and J.~Witaszek).  Several of the authors also worked on aspects of this paper while visiting Oberwolfach in February 2019.  We thank J.~Lipman for comments on a previous draft, particularly pointing out to us the relation of \cite{KleinLecturesIcosahedron} to \autoref{thm.CoversE8}.  We also thank B.~Martinova for valuable comments on a previous draft.  Finally,  we are very thankful to the anonymous referee for several useful comments and suggestions,  particularly for making us aware of the work \cite{GreuelKroningSimpleSingularitiesInPosChar} and that \autoref{lem.SplittingLemma} is a weak version of what is known as the ``splitting lemma'' among experts (see \autoref{remark:splittinglemmaremark}) in equal characteristic zero.

\section{Prelimaries}

\subsection{Surface singularities}

All rings and schemes in this article are assumed to be excellent.  Every excellent surface, independent of characteristic, admits a resolution of singularities, \cite{LipmanRationalSingularities, LipmanDesingularizationOf2Dimensional}. Of particular interest are surfaces with rational singularities. A surface $X$ has rational singularities if for some, equivalently all, resolution of singularities $Y\xrightarrow{\pi} X$ has the property that $R^1\pi_* \O_X=0$. If $X$ is a surface with rational singularities then $X$ can be resolved by quadratic transforms, i.e., by blowing up closed singular points one at a time. If $X$ is a surface with rational singularities and $Z\to X$ is a quadratic transform, then $Z$ has rational singularities. In particular, $Z$ is a normal surface, \cite[Proposition~8.1]{LipmanRationalSingularities}. The property that a sequence of quadratic transform preserves normality characterizes rational surface singularities and is fundamental to the proof of Theorem~\ref{Main Theorem description of functions}.

Resolving a rational surface singularity $X$ by quadratic transforms produces a minimal resolution $Y$ of $X$. The dual graph of $X$ is the graph whose vertices correspond to exceptional curves of the minimal resolution of $X$ and edges connect two vertices provided those two curves intersect in $Y$. The possible graphs that can be obtained fall under the classification of the Dynkin diagrams $A_n, D_n, E_6, E_7,$ and $E_8$ and are referred to as the type of $X$.

\subsection{$\mathbb{Q}$-Cartier divisors and cyclic covers}

Let $X$ be a normal scheme. A Weil divisor $D$ of $X$ is said to be $\mathbb{Q}$-Cartier if  there exists a natural number $n>0$ such that $nD$ is a Cartier divisor. Suppose that $D$ is a Weil divisor and $nD \sim 0$. Fix an isomorphism $\O_X \cong \O_X(-nD)$.  Then, the index $n$ cyclic cover of $\O_X$ relative to the $\Q$-Cartier divisor $D$ (and to the fixed isomorphism $\O_X \cong \O_X(-nD)$) is the $\O_X$-algebra
\[
C=\O_X\oplus \O_X(D)\oplus \O_X(2D)\oplus \cdots \oplus \O_X((n\\
-1)D),
\]
where multiplication in $C$ is determined by the natural multiplication maps
\[
\O_X(iD)\otimes \O_X(jD)\to \O_X\big((i+j)D\big)
\]
and the previously fixed isomorphism $\O_X(-nD) \cong \O_X$: if $i+j\geq N$ then our previous isomorphism determines an isomorphism $\O_{X}((i+j)D)\cong \O_{X}((i+j-n)D)$.  Observe that the map $\O_X\to C$ splits as a map of $\O_X$-modules.


Suppose that $(R,\fram,\kay)$ is a local normal domain and $X=\Spec(R)$. If $I\subset R$ is a pure height $1$ ideal, then $I=\O_X(D)$ for some anti-effective divisor $D$ and $\O_X(iD)=I^{(i)}$ is the $i$th symbolic power of $I$ for each $i\in \N$. Therefore, if $nD$ is Cartier, that is if $I^{(n)}=(f)$ is a principal ideal, then the cyclic cover of index $n$ corresponding to $D$ and $f$ is the $R$-algebra
\[
C=R\oplus I\oplus I^{(2)}\oplus \cdots \oplus I^{(n-1)}.
\]
Multiplication in $C$ is determined by the natural multiplication maps $I^{(i)}\otimes I^{(j)}\to I^{(i+j)}$ and isomorphisms $I^{(i+j)}\xrightarrow{\cdot 1/f}I^{(i+j-n)}$ whenever $i+j\geq n$. The order of a Weil divisor is the least natural number $n$ so that $nD$ is Cartier. Cyclic covers of index equal to the order of the Weil divisor are domains.

\begin{proposition}{\cite[Corollary~1.9]{TomariWatanabeNormalZrGradedRings}}
\label{prop.Cyclic covers are domains}
Let $R$ be a normal domain, $X=\Spec(R)$, and $D$ a Weil divisor of order $n$. Then the cyclic cover
\[
C=\O_X\oplus \O_X(D)\oplus \O_X(2D)\oplus \cdots \oplus \O_X((n-1)D)
\]
is a domain.
\end{proposition}

\section{Classification of rational double points in mixed characteristic}
\label{sec.Classification}

Throughout this section, we denote by $(S,\fran,\kay)$ a regular local ring of dimension $3$, $(R,\fram,\kay)=S/(f)$ is a $2$-dimensional quotient of $S$ of multiplicity $2$, and the letters $x,y,z$ will be used to denote a choice of minimal generators of the maximal ideal of $S$ (and $R$). Theorem~\ref{Main Theorem description of functions} is a combination of the results in this section. We attempt to make the results in this section as general as possible and remark that any ring satisfying the hypotheses of Theorem~\ref{Main Theorem description of functions} satisfy the hypotheses of each of the statements found in this section. Our classification techniques are characteristic free in the sense that they do not depend on the characteristic of the ring, but only the on the characteristic of the residue field. Therefore the techniques of this section can be used to study the classification of rational double points in equicharacteristic $0$ and $p>5$.  These techniques should be compared to those in \cite{GreuelKroningSimpleSingularitiesInPosChar}.

\begin{lemma}[{cf. \cite[Theorem 11.1]{ArnoldGuseinZadeVarchenkoSingularitiesOfDifferentiableMaps} and \cite[Theorem 2.47]{GreuelLossenShustinIntroToSingularitiesAndDeformations}}] \label{lem.SplittingLemma}
Let $(S,\fran,\kay)$ be a $3$-dimensional strictly Henselian regular local ring and $\Char \kay >2$. Suppose that $f\in \fran^2-\fran^3$. Then there exists a choice of generators $x,y,z$ of the maximal ideal of $S$ so that $f=\tilde{f}+g$, $g\in \fran^3$, and $\tilde{f}$ is either $x^2+y^2+z^2$, $x^2+y^2$, or $x^2$.
\end{lemma}
\begin{proof}

Choose generators $x,y,z$ of the maximal ideal $\fran$. We begin by writing $f=\tilde{f}+g$ with $g \in \fran^3$, and $\tilde{f}$ is a ``quadratic form,'' in other words
\[
\tilde f = \bm{x}^T U \bm{x}
\]
where $U$ is a symmetric $3 \times 3$ matrix over $S$ whose entries are either $0$ or units and where $\bm{x}^T = \begin{bmatrix}
x&y&z
\end{bmatrix}$.

After reduction modulo $\fran$, we have that the symmetric matrix $U$ is (orthogonally) diagonalizable over $\kay$. This mean that after a linear change of variables we may assume $U$ is diagonal modulo $\fran$. By lifting this to $S$, we get that after choosing new minimal generators of $\fran$, $f$ can be written as
\[
f=ux^2+vy^2+wz^2 + g'
\]
where $u,v,w$ are either $0$ or units, and $g' \in \fran^3$. Moreover, since $S$ is strictly Henselian (and $p \neq 2$), we have that
\[
f=\left(u^{1/2} x\right)^2+\left(v^{1/2}y\right)^2+\left(w^{1/2}z\right)^2 + g'
\]
by extracting square roots of units (or zero). Hence, by a new choice of minimal generators of $\fran$, we obtain the desired result.
\end{proof}

\begin{remark}
\label{remark:splittinglemmaremark}
Over the complex numbers, Lemma~\ref{lem.SplittingLemma} is a weak version of what is known as the ``splitting lemma'' (or even the ``generalized Morse lemma''). See \cite[Theorem 11.1]{ArnoldGuseinZadeVarchenkoSingularitiesOfDifferentiableMaps} and \cite[Theorem 2.47]{GreuelLossenShustinIntroToSingularitiesAndDeformations} for further details; the key distinction above is that we do not know if the parameters can be chosen so that those in the quadratic form do not appear in the collection of higher order terms. In other words and in the notation of Lemma~\ref{lem.SplittingLemma}, it would be interesting to know if one can always find a generating set $x,y,z$ of $\n$ and an expression $f=\tilde{f}+g$ with $g\in \n^3$ so that if $g=\sum u_{i,j,k}x^iy^jz^k$ is an expression of $g$ in the completion $\widehat{S}$, with each $u_{i,j,k}$ either a unit or $0$, then
\begin{enumerate}
\item if $f=x^2$ then $u_{i,j,k}=0$ whenever $i\geq 1$;
\item if $f=x^2+y^2$ then $u_{i,j,k}=0$ whenever $i\geq 1$ or $j\geq 1$;
\item if $f=x^2+y^2+z^2$ then $g=0$.
\end{enumerate}
Such expressions are possible if the characteristic of $\kay$ exceeds $5$ and $f$ defines a rational double point, see Theorem~\ref{Main Theorem description of functions}. In general, one might also hope that the collection of higher order terms satisfies similar uniqueness properties as in the classical setting: if $f=\tilde{f}_1+g_1=\tilde{f}_2+g_2$ are two such expressions in different choices of parameters will it be the case that $g_1, g_2$ are analytically equivalent? We are grateful to the anonymous referee for making us aware of the ``splitting lemma'' references listed above as well as the interesting questions raised in comparison to Lemma~\ref{lem.SplittingLemma}.
\end{remark}

\begin{proposition} \label{prop.ThreeSquares}
Let $(S,\fran, \kay)$ be a $3$-dimensional regular strictly Henselian local ring with $\fran=(x,y,z)$ and $\Char \kay >2$.  Suppose $f\in S$ defines a rational double point and is of the form $f=x^2+y^2+z^2+g$ with $g \in \fran^3$. Then one can choose new minimal generators of $\fran$, say $\fran=\bigl(\tilde{x}, \tilde{y}, \tilde{z}\bigr)$, such that $f=\tilde{x}^2+\tilde{y}^2+\tilde{z}^2$. In other words, $S/(f)$ is of type $\textnormal{A}_1$.
\end{proposition}
\begin{proof}
Say first
\[
g=\sum_{i+j+k=3} \alpha_{ijk}x^iy^jz^k
\]
for some $\alpha_{ijk} \in S$. Then $f$ can be written as
\begin{align*}
f &=  x^2 + \alpha_{300}x^3 + \sum_{j+k=1} \alpha_{2jk}x^2 y^jz^k\\
&+y^2 + \alpha_{030}y^3 + \sum_{i+k=1} \alpha_{i2k}x^i y^2z^k\\
&+z^2 + \alpha_{003}z^3 + \sum_{i+j=1} \alpha_{ij2}x^i y^jz^2 \\
&+\alpha_{111}xyz
\end{align*}
Now, by factoring out the squares in the first three lines, we get that
\[
f= ux^2+vy^2+wz^2 + \alpha_{111}xyz
\]
where $u,v,w$ are units of $S$.\footnote{For instance, $u=1+\alpha_{300}x + \sum_{j+k=1}\alpha_{2jk}y^jz^k$.} Using that $S$ is strictly Henselian and $p\neq 2$, we can extract square roots of $u$, $v$, and $w$ to absorb them into the squares. Then, by declaring new generators of $\fran$, we may assume $g=\alpha xyz$ for some $\alpha \in S$.\footnote{Indeed, we may set the new generators of $\fran$ to be $u^{1/2} x$, $v^{1/2} y$, and $w^{1/2} z$. Thus, $\alpha=\alpha_{111}u^{-1/2}v^{-1/2}w^{-1/2}$.}

Next, we proceed under the assumption $g=\alpha xyz$ for some $\alpha \in S$. There are two cases depending on whether or not $\alpha \in \fran$. If $\alpha \in \fran$, we can just repeat the previous argument with a ``$g$'' that has no ``$\alpha_{111}$.'' Hence, this would lead to the case $g=0$ after a new choice of generators of $\fran$. Then, we may assume $\alpha$ is a unit, in which case:
\begin{equation} \label{eqn.squaretrick}
\begin{split}
f&=x^2+y^2+z^2+\alpha xyz\\
&=(x+y)^2+z^2+xy(\alpha z-2)\\
&= (x+y)^2+z^2+4^{-1}(\alpha z -2)\bigl((x+y)^2-(x-y)^2\bigr)\\
&= \bigl(1+4^{-1}(\alpha z -2)\bigr)(x+y)^2+4^{-1}(\alpha z -2)(x-y)^2 +z^2.
\end{split}
\end{equation}
Let $\mu= 4^{-1}(\alpha z -2)$. We remark that both $\mu$ and $1+\mu$ are units of $S$ (we use that $\Char \kay > 2$ to deduce that $1+\mu$ is a unit).

Observe that the maximal ideal of $S$ is minimally generated by $x+y,x-y,z$. Therefore, as before, we may assume that there are minimal generators $\tilde{x},\tilde{y},\tilde{z}$ of the maximal ideal of $S$ such that $f=\tilde{x}^2+\tilde{y}^2+\tilde{z}^2$. \end{proof}

\begin{proposition}
\label{prop.AnClassification}
Let $(S,\fran, \kay)$ be a $3$-dimensional regular strictly Henselian local ring with $\fran=(x,y,z)$ and $\Char \kay >2$.  Suppose $f\in S$ defines a rational double point and is of the form $f = x^2+y^2+g$ with $g \in \fran^3$. Then one can choose new minimal generators of $\fran$, say $\fran=\bigl(\tilde{x}, \tilde{y}, \tilde{z}\bigr)$, such that, up to a unit, $f=\tilde{x}^2+\tilde{y}^2+\tilde{z}^{n+1}$ with $n \geq 2$. In other words, $S/(f)$ is of type $\textnormal{A}_n$ for some $n \geq 2$.
\end{proposition}
\begin{proof}

Following \autoref{prop.ThreeSquares}, we write
\[
f=x^2+y^2+\sum_{i+j+k=3} \alpha_{ijk}x^iy^jz^k
\]
for some $\alpha_{ijk} \in S$. Then by grouping together the terms with $x^2$ and $y^2$ factors we have that
\[
f=ux^2+vy^2+ z^2\bigl(\alpha_{102}x+\alpha_{012}y+\alpha_{003}z\bigr)+  \alpha_{111}xyz.
\]
Moreover, after absorbing the units $u,v$ into the squares by a new choice of generators, we have
\[
f=x^2+y^2 +z^2(\alpha_1 x + \alpha_2 y + \alpha_3 z) + \alpha_0 xyz
\]
for some $\alpha_i \in S $. By using the idea from \autoref{eqn.squaretrick}, we may absorb the term $\alpha_0 xyz$ into the sum of squares $x^2+y^2$ by choosing new generators of $\fran$. Thus, we may assume
\[
f=x^2+y^2 +z^2(\alpha x + \beta y + \gamma z)
\]
for some $\alpha, \beta, \gamma \in S$. By completing squares, we also have
\begin{align*}
f&=\left(x+\frac{\alpha}{2}z^2\right)^2+\left(y+\frac{\beta}{2}z^2\right)^2+\gamma z^3 -\left(\frac{\alpha^2}{4}+ \frac{\beta^2}{4}\right) z^4\\
&=\left(x+\frac{\alpha}{2}z^2\right)^2+\left(y+\frac{\beta}{2}z^2\right)^2+\delta z^3
\end{align*}
for some $\delta \in S$. Hence, by a new choice of generators of $\fran$, we may assume
\[
f=x^2+y^2+\delta z^3.
\]

If $\delta$ is a unit, then $S/(f)$ is of type $\textnormal{A}_2$. Indeed, we have that
\[
\delta^{-1}f=\delta^{-1}x^2+\delta^{-1}y^2+z^3
\]
and then we can let $x$ and $y$ to absorb the units $\delta^{-1}$ by taking their square root and choosing new generators.

If $\delta$ is not a unit however, write $\delta=\lambda x + \mu y +\nu z$ and complete squares once again to say
\[
f=x^2+y^2+\varepsilon z^4
\]
for some $\varepsilon \in S$. If $\varepsilon$ is a unit we then get an $\textnormal{A}_3$ equation (by considering $\varepsilon^{-1}f$ instead as before). Otherwise, we repeat all over again. This process will eventually stop yielding that, up to a choice of generators,
\begin{equation} \label{eqn.Sequence}
f=x^2+y^2 + z^{n+1}
\end{equation}
for some integer $n \geq 2$.

To see why the above described process stops, suppose by sake of contradiction that it does not.  We can also view this process inside the completion $\widehat{S} \supset S$.  Then we end up with a sequence of equations
\[
f = x_n^2+y_n^2+\delta_n z^{n+1} \quad{(n \geq 1)}
\]
where $\delta_n \in \fran$ for all $n$, and also
\[
x_{n+1} = x_n + \frac{\alpha_n}{2} z^n, \quad  y_{n+1} = y_n + \frac{\beta_n}{2} z^n
\]
for some $\alpha_n, \beta_n \in S$.

Hence, the sequences $\{x_n\}$ and $\{y_n\}$ are Cauchy and then converge, say to $\bar{x}$ and $\bar{y}$ respectively. Moreover, from $\autoref{eqn.Sequence}$, we see that the sequence $\bigl\{x_n^2 + y_n^2 \bigr\}$ converges to $f \in \widehat{S}$. Putting these two observations together we conclude that
\[
f = \bar{x}^2 + \bar{y}^2 \in \widehat{S}
\]
where actually $(\bar{x},\bar{y},z) = \fran \widehat{S}$. Then $\widehat{S}/(f)$ is not even normal and so neither is $S/(f)$,\footnote{Recall that we assume $R$, $S$ are excellent rings throughout the article.} which is a contradiction.
\end{proof}

The remaining singularities classifying rational double points will be defined by functions of the form $f=x^2+g$ with $g\in \fran^3$. We continue with a lemma which further refines the possible forms of the function $f$.

\begin{lemma}
\label{lem:Three possible forms}
Let $(S,\fran,\kay)$ be a $3$-dimensional regular local ring. Suppose that either $S$ is strictly Henselian with $\Char \kay >3$ or that $S$ is Henselian with algebraically closed residue field of prime characteristic $p>2$.  Suppose $f\in S$ defines a rational double point and is of the form $f=x^2+g$ with $g\in \fran^3$. Then there exists a possibly new choice of generators $x,y,z$ of the maximal ideal of $S$ such that $f$ is of one of the following forms:
\begin{enumerate}
\item $f=x^2+y^2z+z^3+h$ with $h\in (y,z)^4$;
\item $f=x^2+y^2z+h$ with $h\in (y,z)^4$;
\item $f=x^2+y^3+h$ with $h\in (y,z)^4$.
\end{enumerate}
\end{lemma}

\begin{proof}
We start by writing
\[
f=x^2+\sum_{i+j+k=3} \alpha_{ijk}x^iy^jz^k
\]
for some $\alpha_{ijk} \in S$.  Similar to the proof of \autoref{prop.AnClassification}, by grouping terms, we may write
\[
f=x^2+y^2 \cdot \alpha + z^2 \cdot \beta
\]
 where $\alpha, \beta \in \fran$.
 Say now $\alpha= \alpha_1 x + \alpha_2 y + \alpha_3 z$ and $\beta= \beta_1 x + \beta_2 y + \beta_3 z$. 
Next, by completing the squares, we can group all appearances of $x$ in $\alpha$ or $\beta$ in one single square. Thus, we may assume
\[
f=x^2 +  ay^3 + by^2z + cyz^2+dz^3
\]
where $a,b,c,d$ are units or belong to the ideal $(y,z)$. Moreover, at least one of these must be a unit. Else, $f$ is of the form $f=x^2+g$ where $g\in \n^4$. However, one readily verifies that the blowup of $S/(f)$ at the closed point produces a non-normal scheme,\footnote{Indeed, consider the chart
$T = S[\tld{x}, \tld{y}, \tld{z}] = S[x/y, y, y/z]$ where the strict transform of $V(f)$ is defined by $\tld f = \tld x^2 + g/y^2$ where we observe that $g/y^2 \in (y^2)_{S_y}$ (since $g \in \n^4$).  Now, $(\tld x, y)$ generates a height 2 prime in $T$ and so a height 1 prime $Q$ in $T/(\tld f)$.  However, $\tld f \in Q^2$, and so $T/(\tld f)$ is not regular in codimension 1.} which contradicts \cite[Proposition~8.1]{LipmanRationalSingularities}. Furthermore, we may assume all $a,b,c,d$ are units or zero if instead we write
\[
f=x^2 +  ay^3 + by^2z + cyz^2+dz^3 + h
\]
where $h \in (y,z)^4$. Notice that if $a$ and $d$ are zero, then the above ``cubic form'' factors as
\[
ay^3 + by^2z + cyz^2+dz^3 = yz(by+cz)
\]
which is a product of three ``linear forms.'' Otherwise, suppose without lost of generality that $a$ is a unit, say $a=1$. We then consider the \emph{cubic} polynomial $p(t)=t^3+bt^2+ct+d \in \kay[t]$. Recall that we are assuming $\kay$ is either separably closed of characteristic $p>3$ or that $\kay$ is algebraically closed of characteristic $p>2$. In either case, $p(t)$ factors as a product of linear factors in $\kay[t]$; see \cite[Page 184]{RomanFieldTheory}. 
Therefore, at the residue field level, $p(t)$ admits a factorization
\[
p(t) = (t-\lambda_1)(t-\lambda_2)(t-\lambda_3) \in \kay[t].
\]
Lifting the $\lambda_i$ back to $S$ (in some arbitrary way), we obtain that
\[
ay^3 + by^2z + cyz^2+dz^3  - (y-\lambda_1 z)(y-\lambda_2 z) (y-\lambda_2 z) \in   (y,z)^3\cdot(x,y,z).
\]
In other words, by a new choice of ``$h\in (y,z)^4$'' if necessary, we may assume
\[
f=x^2 + \ell_1 \ell_2 \ell_3 + \gamma x+h
\]
where the $\ell_i$ are  ``linear forms,'' $\gamma\in (y,z)^3$, and $h\in (y,z)^4$. Observe that $x^2+\gamma x=(x+2^{-1}\gamma)^2-4^{-1}\gamma^4$. Therefore, we may replace $x$ by $x+2^{-1}\gamma$, $h$ by $h-4^{-1}\gamma^4$, and assume further that
\[
f=x^2 + \ell_1 \ell_2 \ell_3 +h
\]
with $h\in (y,z)^4$. There are three cases to consider.

\noindent \underline{Case 1}: $\ell_1, \ell_2, \ell_2$ define different lines.\footnote{Meaning that these are different lines in the $\kay$-plane generated by $y,z$ in the cotangent space $\fran/\fran^2$.}

Then there exists a $\kay$-basis for $\fran/\fran^2$, say $x, \tilde{y}, \tilde{z}$, such that $\ell_1 = \tilde{z}$, $\ell_2 = \tilde{z}+i\tilde{y}$, and $\ell_3 = \tilde{z}-i\tilde{y}$ in that plane.\footnote{Indeed, we may assume $\ell_3 = a_1 \ell_1 + a_2\ell_2 $ in the $\kay$-plane spanned by $y,z$ in $\fran/\fran^2$, this for some $0 \neq a_i \in \kay$. Then by relabeling $\ell_i$ by $a_i \ell_i$, we may assume $\ell_3 = \ell_1 + \ell_2$, or even better $\ell_3 = \frac{1}{2} \ell_1 + \frac{1}{2}\ell_2$. Then we may choose our coordinates for this plane, say $\tilde{y}, \tilde{z}$, so that $\ell_1$ and $\ell_2$ are the orthogonal lines $\tilde{z}+i\tilde{y}$ and $\tilde{z}-i\tilde{y}$ respectively. Then we would have $\ell_3 = \tilde{z}$.} Now, when this is lifted back to $S$, we simply obtain that $\ell_1 \ell_2 \ell_3 - \tilde{z}\bigl(\tilde{z}+i\tilde{y}\bigr)(\tilde{z}-i\tilde{y}) \in \fran^4$. Thus, we may assume that
\[
f=x^2+y^2z+z^3+h
\]
and $h \in (y,z)^4$.

\noindent \underline{Case 2}: $\ell_1 \neq \ell_2 =\ell_3$.

Then we may assume $\ell_2 = \ell_3 = y$ and $\ell_1=z$. This gives that
\[
f=x^2+y^2z+ h
\]
with $h\in (y,z)^4$.

\noindent \underline{Case 3}: $\ell_1,\ell_2,\ell_3$ are the same.

Then we may assume $\ell_1=y$ and $f$ is of the form $f=x^2+y^3+h$ with $h\in (y,z)^4$.
\end{proof}

\begin{proposition}
\label{prop:D4 classification} Let $(S,\fran,\kay)$ be a $3$-dimensional complete regular local ring of with separably closed residue field of characteristic $p > 2$. Suppose $f\in S$ defines a rational double point and, up to a unit, is of the form $f=x^2+y^2z+z^3+h$ with $h\in (y,z)^4$.  Then, there exists a choice of minimal generators $\tilde{x},\tilde{y}, \tilde{z}$ of the maximal ideal of $S$ such that $f=\tilde{x}^2+\tilde{y}^2\tilde{z}+\tilde{z}^3$, that is $R=S/(f)$ is of type $\textnormal{D}_4$.
\end{proposition}

\begin{proof}
Observe that $h$ is an $S$-linear combination of $y^4$, $y^3z$, $y^2z^2$, $yz^3$, and $z^4$. Then,
\[
f=x^2+uy^2z+vz^3+\alpha y^4
\]
where $u,v$ are units and $\alpha \in S$. Using $y^2$ to absorb the unit $u$, and then multiplying by $v^{-1}$ and using $x^2$, $y^2$ to absorb $v^{-1}$, we may assume that, after replacing $f$ by a unit multiple of $f$, we have
\[
f=x^2+y^2z+z^3+\alpha y^4.
\]
If $\alpha = 0$, we are done. Otherwise, notice we may express $f$ as
\[
f=x^2+y^2\big(z+\alpha y^2\big) + z^3.
\]
If we let $z_1 = z+\alpha y^2$, we then have
\begin{align*}
f=x^2+y^2z_1 + \bigl(z_1-\alpha y^2\bigr)^3 &= x^2+y^2z_1 +z_1^3-3\alpha z_1^2 y^2+3\alpha^2z_1y^4 - \alpha^3y^6\\
&=x^2+u_1y^2z_1 + z_1^3 - \alpha^3 y^6\\
&=x^2+y_1^2z_1 + z_1^3 + \alpha_1 y_1^6,
\end{align*}
where $u_1 = 1 - 3 \alpha z_1 + 3 \alpha^2 y^2$ is a unit,
$y_1 = u_1^{1/2}y$ where we choose $u_1^{1/2} \equiv_{\fran} 1$, and $\alpha_1 = -\alpha^3 u_1^{-3}$. Since $\alpha_1 \neq 0$, we let $z_2 = z_1+\alpha_1 y_1^4$ and write
\begin{align*}
f=x^2+y_1^2z_2 + \bigl(z_2-\alpha_1 y_1^4\bigr)^3 &= x^2+y_1^2z_2 +z_2^3-3\alpha_1 z_2^2 y_1^4+3\alpha_1^2z_2y_1^8 - \alpha_1^3y_1^{12}\\
&=x^2+u_2y_1^2z_2 + z_2^3 - \alpha_1^3 y_1^{12}\\
&=x^2+y_2^2z_2 + z_2^3 + \alpha_2 y_2^{12}
\end{align*}
where $u_2 = 1 - 3 \alpha_1 z_2 y_1^2 + 3 \alpha_1^2 y_1^6$ is a unit,
$y_2 = u_2^{1/2}y_1$ where we choose $u_2^{1/2} \equiv_{\fran} 1$, and $\alpha_2 = -\alpha_1^3 u_2^{-3}$.

Note this process can be repeated indefinitely yielding sequences $\{z_n\}$, $\{y_n\}$, $\{u_n\}$ and $\{\alpha_n\}$, as well as equations
\begin{equation} \label{sequenceEQNS}
f=x^2+y_n^2 z_n +z_n^3 + \alpha_n y_n^{e_n}
\end{equation}
where the exponents $e_n$ are strictly increasing: indeed, one readily checks that they satisfy the recursive formula $e_{n+1}=3(e_n -2)$ with $e_0=4$. Certainly, $\alpha_n$ is never zero. However, the sequence of equations \autoref{sequenceEQNS} implies that the sequence $\bigl\{x^2+y_n^2z_n+z_n^3\bigr\}$ converges to $f$. On the other hand, we notice that $z_{n+1}=z_n+\alpha_n y_n^{e_n-3}$. Therefore, the sequence $\{z_n\}$ is Cauchy and so convergent, denote its limit by $\bar{z}$.

We also claim the sequence $\{y_n\}$ is Cauchy. To show this, we first write:
\[
u_n = 1 - 3 \alpha_{n-1} z_n y_{n-1}^{e_{n-1}-4}   + 3 \alpha_{n-1}^2 y_{n-1}^{2(e_{n-1}-3)},
\]
whereby the units $u_n$ are so that
\[
u_{n+1}-1 \in \bigl(y_{n}^{e_n-4} \bigr).
\]
Hence, from the relation $y_{n+1}^2=u_{n+1}y_n^2$, we have
\[
(y_{n+1}-y_n)(y_{n+1}+y_n) = y^2_{n+1}-y^2_n = (u_{n+1} - 1)y_{n}^2 \in \bigl(y_{n}^{e_n-2}\bigr) \subset \fran^{e_n-2}.
\]
Next, notice that, since we chose $u_{n+1}^{1/2} \equiv_{\fran} 1$, we have that $u_{n+1}^{1/2}+1$ is a unit since $\Char \kay >2$.
Nonetheless, we note that by induction on $n$, the sum $y_{n+1}+y_n = (u_{n+1}^{1/2} + 1)y_n$ belongs to $\fran \smallsetminus \fran^2$.
Consequently, $y_{n+1}-y_n \in \fran^{e_n-2}$. Hence, the sequence $\{y_n\}$ is Cauchy. Let $\bar{y}$ be the limit of $\{y_n\}$.

By taking limits, we then obtain:
\[
f=x^2 + \bar{y}^2 \bar{z}+\bar{z}^3
\]
which is an equation of type $\textnormal{D}_4$, as desired.
\end{proof}

\begin{proposition}
\label{prop:Dn classification} Let $(S,\fran,\kay)$ be a $3$-dimensional complete regular local ring with separably closed residue field of characteristic $p > 2$. Suppose $f\in S$ defines a rational double point and is of the form $f=x^2+y^2z+h$ with $h\in (y,z)^4$.  Then, there exists a choice of minimal generators $\tilde{x},\tilde{y}, \tilde{z}$ of the maximal ideal of $S$ and an integer $n\geq 5$ such that $f=\tilde{x}^2+\tilde{y}^2\tilde{z}+\tilde{z}^{n-1}$, that is $R=S/(f)$ is of type $\textnormal{D}_n$ for some $n\geq 5$.
\end{proposition}

\begin{proof}
As before, we note that $h$ is an $S$-linear combination of $y^4$, $y^3z$, $y^2z^2$, $yz^3$, and $z^4$. Then, by the above arguments and simplifications, we may assume
\[
f=x^2+y^2z+\alpha y^4+ \beta z^3
\]
for some $\alpha \in S$ and $\beta \in (y,z)$. Next, we may use an identical argument to the one in \autoref{prop:D4 classification} (where we had $\beta=1$) to show that we may assume $\alpha=0$. Since the argument is essentially the same, we isolate it in the following claim yet provide it for sake of completeness (pointing out the differences).
\begin{claim}
We may assume that $\alpha=0$.
\end{claim}
\begin{proof}[Proof of claim]
Indeed, we set $z_0 \coloneqq z$, $y_0\coloneqq y$, $\alpha_0\coloneqq \alpha$, $u_0\coloneqq 1$, and $e_0 \coloneqq 4$. Then, we define sequences $\{e_n\},\{z_n\}$, $\{y_n\}$, $\{\alpha_n\}$, and $\{u_n\}$ recursively as follows. First, we define $e_{n+1} \coloneqq 3(e_n-2)$ (which is a strictly inceasing sequence of even positive integers), $z_{n+1}\coloneqq z_n+\alpha_n y_n^{e_n-2}$. Next, we define the units,
\[
 u_{n+1} \coloneqq 1-3\beta \alpha_{n-1}z_ny_{n-1}^{e_{n-1}-4}+3\beta \alpha_{n-1}^2y_{n-1}^{2(e_{n-1}-3)}.
\]
Finally, we set  $ y_{n+1} \coloneqq u_{n+1}^{1/2}y_{n}$ where $u_{n+1}^{1/2} \equiv 1 \bmod \fran$, and $\alpha_{n+1}=-\beta \alpha_n^{3}u_{n+1}^{-e_{n+1}/2}$.

With those definitions in place, we have:
\[
f=x^2+y_n^2z_n+\beta z_n^3+\alpha_n y^{e_n}_n
\]
for all $n\geq 0$.

Of course, this process recovers the sequences in the proof of \autoref{prop:D4 classification} by specializing to $\beta=1$. Just as in the proof of \autoref{prop:D4 classification}, we have that both $\{z_n\}$ and $\{y_n\}$ (as well as $\{u_n\}$) are Cauchy sequences. By taking limits, we obtain the desired statement (observing $\beta$ stayed unchanged).
\end{proof}

Assuming $\alpha=0$, and writing $\beta = \alpha_1 y + \beta_1 z$, we have
\begin{align*}
f=x^2+y^2  z +  \alpha_1 y z^3 + \beta_1 z^4
&= x^2+ \bigl(y^2  + \alpha_1 y z^2\bigr) z + \beta_1 z^4 \\
&=x^2+\left(y  + \frac{\alpha_1}{2} z^2\right)^2 z+ \gamma_1 z^4
\end{align*}
where in the last step we completed the square and for this we may need to add an extra multiple of $z^5$ that was absorbed by the term $\beta_1 z^4$ under a new coefficient $\gamma_1$. In conclusion, we may assume
\[
f=x^2+y^2z+\gamma_1 z^4
\]
for some $\gamma_1 \in S$. If $\gamma_1$ is a unit then $f$, or more precisely $\gamma_1^{-1}f$ (after absorbing $\gamma^{-1/2}$ into $x$ and $y$),
gives an equation of type $\textnormal{D}_5$. Otherwise, we write $\gamma_1 = \alpha_2 y + \beta_2 z$ and we obtain
\begin{align*}
f=x^2+y^2  z + \alpha_2 y z^4 + \beta_2 z^5
&= x^2+\bigl(y^2  + \alpha_2 y z^3\bigr) z + \beta_1 z^5 \\
&=x^2+\left(y  + \frac{\alpha_2}{2} z^3\right)^2 z + \gamma_2 z^5\\
&=x^2+y_1^2z+\gamma_2 z^5
\end{align*}
for $y_1=y  + (\alpha_2/2) z^3$. If $\gamma_2$ is a unit, we end up with a singularity of type $\textnormal{D}_6$. Otherwise, we repeat this process until we get an equation of type $\textnormal{D}_n$ for $n \geq 5$. We remark this process must stop because else we would get a Cauchy sequence $\{y_n\}$, whose limit is say $\bar{y}$, so that the sequence $\{x^2+y_n^2 z\}$ converges to $f$. Then, we would have $(x, \bar{y}, z) = \fran$ such that
\[
f=x^2+\bar{y}^2z
\]
this, however, is not a rational singularity. Indeed, this is not even normal, for $x/\bar{y} \in \Frac R$ is integral over $R$ but not in $R$.
\end{proof}

\begin{proposition}
\label{prop:E6,E7,E8}
Let $(S,\fran,\kay)$ be a $3$-dimensional strictly Henselian regular local ring  with residue field of characteristic $p > 5$.  Suppose that $f\in S$ defines a rational double point and, up to a unit, is of the form $f=x^2+y^3+h$ with $h\in (y,z)^4$. Then one of the following cases occurs.

\begin{enumerate}
\item There exists a choice of minimal generators $\tilde{x}, \tilde{y},\tilde{z}$ of the maximal ideal of $S$ such that $f$ is of one of the following forms:
\begin{enumerate}
\item $f=\tilde{x}^2+\tilde{y}^3+\tilde{z}^4$;
\item $f=\tilde{x}^2+\tilde{y}^3+\tilde{z}^5$.
\end{enumerate}
\item There exists a choice of minimal generators $\tilde{x}, \tilde{y},\tilde{z}$ of the maximal ideal of $\widehat{S}$ such that in $\widehat{S}$, $f$ is of the form $f=\tilde{x}^2+\tilde{y}^3+\tilde{y}\tilde{z}^3$.
\end{enumerate}
Therefore $R=S/(f)$ is of type $\textnormal{E}_6, \textnormal{E}_7$, or $\textnormal{E}_8$.
\end{proposition}

\begin{proof}
We may assume
\[
f=x^2+y^3 + 3\alpha y^2z^2 + \beta y z^3 + \gamma z^4
\]
for some $\alpha, \beta , \gamma \in S$ (absorbing units into $y$ as necessary). Our first observation is that by completing the cube one gets
\[
f=x^2+\bigl(y+\alpha z^2\bigr)^3 + \beta y z^3 + \gamma_1 z^4
\]
for some new $\gamma_1$. Then by declaring $y_1 = y+\alpha z^2$ we have that
\[
f=x^2 + y_1^3 + \beta y_1 z^3 + \gamma_2 z^4
\]
for some other $\gamma_2$. It is worth mentioning that the completion of the cube did not require the term $\beta yz^3$ to be changed, this will be an important observation paragraphs below. This means that we may assume $\alpha = 0$, or in other words that
\begin{equation} \label{StartingEquationE}
f=x^2+y^3+\delta yz^3 + \varepsilon z^4.
\end{equation}

\subsection*{$\varepsilon$ is a unit, the $\textnormal{E}_6$ case}
Suppose $\varepsilon$ is a unit, then $\varepsilon^{-1} f$ can be written as
\[
\varepsilon^{-1}f=x^2+y^3+4\delta_1 yz^3 + z^4
\]
after a new choice of generators of the maximal ideal. Thus, we may ``complete the tesseract'' to get
\[
\varepsilon^{-1}f=x^2+uy^3 - 6\delta_1^2 y^2z^2 +  (\delta_1 y + z)^4
\]
where $u \in S$ is a unit. By letting $z_1 = \delta_1 y + z$, we have
\begin{align*}
\varepsilon^{-1}f &=  x^2+uy^3 - 6\delta_1^2 y^2(z_1-\delta_1 y)^2 +  z_1^4 \\
&=  x^2+uy^3 - 6\delta_1^2 y^2(z_1-\delta_1 y)^2 +  z_1^4 \\
&=  x^2+vy^3 + \delta_2 y^2 z_1^2 +z_1^4
\end{align*}
with $v$ a unit and some $\delta_2 \in S$. Now, after absorption of the unit $v$ into $y$ (here we need $p>3$) and completing the cubes, we obtain, after a new choice of ``$y$,'' that
\[
\varepsilon^{-1}f=x^2 + y_1^3 + w z_1^4
\]
for some unit $w$. Absorbing $w$ into $z_1$, we see that $\varepsilon^{-1}f$ is of type $\textnormal{E}_6$.
\subsection*{$\varepsilon$ is \emph{not} a unit}
Suppose now that $\varepsilon$ is not a unit in \autoref{StartingEquationE}, then we can rewrite that equation as
\[
f=x^2 + y^3 + \rho y z^3 + \sigma z^5 + \kappa x z^4
\]
for some $\rho, \kappa, \sigma \in S$.  Relabeling $x + (\kappa/2) z^4$ as $x$, we may assume that $\kappa = 0$ (absorbing additional terms into $\sigma$), and obtain:
\begin{equation} \label{SecondEquationEtype}
f=x^2 + y^3 + \rho y z^3 + \sigma z^5.
\end{equation}

\subsection*{$\rho$ is a unit, the $\textnormal{E}_7$ case}
Suppose now that $\rho $ is a unit and replace $S$ by its completion. By replacing $z$ with $\rho^{-1/3} z$, we may assume
\[
f= x^2+y^3 + z^3(y+\sigma_1 z^2).
\]
Let $\tilde{y} = y+ \sigma_1 z^2$, then this equation becomes
\begin{align*}
f &=  x^2+\bigl(\tilde{y}-\sigma_1 z^2\bigr)^3 + \tilde{y} z^3 \\
&= x^2+ \tilde{y}^3 -3 \sigma_1 \tilde{y}^2 z^2 + u \tilde{y} z^3 - \sigma_1^3 z^6
\end{align*}
for some unit $u$. 
We write $\tilde{z} = z - (\sigma_1/u) \tilde{y}$ and so after expanding obtain
\begin{align*}
f &=  x^2+ \tilde{y}^3 -3 \sigma_1 \tilde{y}^2 \big(\tilde{z} + (\sigma_1/u) \tilde{y}\big)^2 + u \tilde{y} \big(\tilde{z} + (\sigma_1/u) \tilde{y}\big)^3 - \sigma_1^3 \big(\tilde{z} + (\sigma_1/u) \tilde{y}\big)^6\\
 & =  x^2 + v_1\tilde y^3 + u_1 \tilde y \tilde z^3 - \sigma_1^3 \tilde z^6.
\end{align*}
for some units $u_1$ and $v_1$ (one can check this easily, with for instance Macaulay2, \cite{M2}).
Let $\tilde{y}$ absorb $v_1$, and $\tilde{z}$ absorb $u_1$, so that we may assume
\[
 f = x^2 + y^3 + yz^3 + \tau z^6 = x^2 + y^3  + z^3\bigr(y+\tau z^3\bigl)
\]
for some $\tau \in S$. Let $y_1 = y + \tau z^3$, so that
\begin{align*}
f = x^2+\bigl(y_1-\tau z^3\bigr)^3 + y_1 z^3 &=x^2+ y_1^3 -3 \tau y_1^2 z^3 + 3 \tau^2 y_1 z^6 - \tau^3 z^9 + y_1 z^3\\
& = x^2 + y_1^3 + u_1 y_1 z^3 -\tau^3 z^9 \\
& = x^2 + y_1^3 +  y_1 z_1^3 +\tau_1 z_1^9\\
& = x^2 + y_1^3 + z_1^3\bigl(y_1 + \tau_1 z_1^ 6\bigr)
\end{align*}
where $u_1$ is a unit satisfying $u_1 -1 \in (y_1, z^3)$, $z_1 = u_1^{1/3} z$, and we have chosen the cubic root $u_1^{1/3}$ so that $u_1^{1/3}- 1 \in \fran$. Letting $y_2 = y_1 + \tau_1 z_1 ^6$, a similar computation yields
\begin{align*}
f = x^2+\bigl(y_2-\tau_1 z_1^6\bigr)^3 + y_2 z_1^3 &=x^2+ y_2^3 -3 \tau_1 y_2^2 z_1^6 + 3 \tau_1^2 y_2 z_1^{12} - \tau_1^3 z_1^{18} + y_2 z_1^3\\
& = x^2 + y_2^3 + u_2 y_2 z_1^3 -\tau_1^3 z_1^{18} \\
& = x^2 + y_2^3 +  y_2 z_2^3 +\tau_2 z_2^{18}\\
& = x^2 + y_2^3 + z_2^3\bigl(y_2 + \tau_2 z_2^ {15}\bigr)
\end{align*}
for some unit $u_2$ satisfying $u_2 -1 \in \bigl(y_2 z_1^3, z_1^{9}\bigr)$, $z_2 = u_2^{1/3} z_1$, and we have chosen $u_2^{1/3}$ so that $u_2^{1/3}- 1 \in \fran$. This process can be repeated inductively yielding sequences $\{y_i\}$, $\{z_i\}$, $\{\tau_i\}$, $\{u_i\}$ in $S$, and a sequence $\{e_i\}$ in $\mathbb{N}$, so that $y_{i+1} = y_i + \tau_i z_i^{e_i}$, $e_{i+1} = 3e_i-3$, $e_0=3$, $z_{i+1}=u_{i+1}^{1/3} z_i$, $u_{i+1}-1 \in \bigl(y_{i+1} z_i^{e_i - 3 }, z_i^{2e_i-3}\bigr) \subset \fran^{e_i - 2}$, and $u_{i+1}^{1/3}$ is always chosen so that $u_{i+1}^{1/3} -1 \in \fran$. Furthermore, we have that
\[
f = x^2 + y_i^3 + y_i z_i^3 + \tau_i z_i^ {e_i+3}
\]
since the sequence $\{e_i\}$ is strictly increasing, we have that the sequence $\bigl\{x^2 + y_i^3 + y_i z_i^3 \bigr\}$ converges to $f$. Moreover, we observe that both sequences $\{y_i\}$, $\{z_i\}$ are Cauchy. To see that the sequence $\{z_i\}$ is Cauchy, notice that
\[
\bigl(z_{i+1} - z_i\bigr)\bigl(z_{i+1}^2 + z_{i+1}z_i + z_i^2 \bigr) = z_{i+1}^3 - z_i^3 =z_i^3(u_{i+1}-1) \in \bigl(y_{i+1} z_i^{e_i }, z_i^{2e_i}\bigr) \subset \fran^{e_1 + 1}
 \]
 due to $u_{i+1}-1 \in \bigl(y_{i+1} z_i^{e_i - 3 }, z_i^{2e_i-3}\bigr) \subset \fran^{e_i - 2}$. Nonetheless, we claim that
 \[
 z_{i+1}^2 + z_{i+1}z_i + z_i^2 = z_i^2 \Bigl(u_{i+1}^{2/3} + u_{i+1}^{1/3} +1 \Bigr) \in \fran^2\smallsetminus \fran^3,
 \]
 or in other words that $u_{i+1}^{2/3} + u_{i+1}^{1/3} +1$ is a unit. However, $u_{i+1}^{1/3}-1 \in \fran$ and $u_{i+1}^{2/3} + u_{i+1}^{1/3} +1 = (u_{i+1}^{2/3} - 1) + (u_{i+1}^{1/3} - 1) + 3$ which is clearly a unit.   Therefore,
 \[
 z_{i+1}-z_i \in \fran^{e_i+1}:\fran^2
 \]
 and so $\{z_i\}$ is Cauchy (because $\bigcap_{e_i} (\fran^{e_i+1}:\fran^2) =0$).  Hence, we may write
 \[
 f = x^2 + y^3 + y z^3
 \]
 and $f$ is $\textnormal{E}_7$.

\subsection*{$\rho$ is not a unit, the $\textnormal{E}_8$ case}

Finally, we return to the strictly Henselian scenario and equation \autoref{SecondEquationEtype} and consider the case $\rho$ is not a unit. Then we may assume
\[
f=x^2 + y^3 + \rho_0 xyz^3 + \rho_1 y^2z^3 +\rho_2 y z^4 + \sigma z^5
\]
for some $\rho_i, \sigma \in S$. Completing the square, or in other words relabeling $x + (\rho_0/2) y z^3$ as $x$, we may assume that $\rho_0 = 0$ (absorbing terms into $\rho_1$).
Additionally, by completing the cube, relabeling $y+(\rho_1/3) z^3$ as $y$, we may likewise assume $\rho_1 = 0$ and so obtain
\begin{equation}
\label{eq.E8BeforeSimplified}
f=x^2 + y^3 + \rho_2 y z^4 + \sigma' z^5.
\end{equation}

If $\sigma'$ is unit (and for simplicity of notation, we call it $\sigma$ also), then
\[
\sigma^{-1} f = x^2 + y^3 +  5 \varrho_2 y z^4 + z^5
\]
for a new choice of generators of $\fran$ and possibly changing $\varrho_2$ up to a unit (here we are also using that $5$ is a unit).  Then, we may complete the quintic in $\sigma^{-1} f$ to obtain
\[
\sigma^{-1} f= x^2 + u y^3 + v y^2 z^3 + \bigl(\varrho_2 y + z\bigr)^5
\]
for some unit $u$ and some $v, \eta \in S$.  Relabeling $\varrho_2 y + z$ as $z$, we obtain
\[
\sigma^{-1} f= x^2 + u y^3 + v y^2 (z - \varrho_2 y)^3 + z^5 = x^2 + u' y^3 + v y^2 z^3 + z^5
\]
Completing the cube again, we obtain
\[
\sigma^{-1} f= x^2 + u'' y^3 + z^5.
\]
Hence, $\sigma^{-1} f$ is of type $\textnormal{E}_8$.

Finally, suppose $\sigma'$ is not unit.  From \eqref{eq.E8BeforeSimplified}, we obtain
\[
f=x^2+y^3+ \gamma_1 yz^4 +\sigma_1 z^6 + \nu x z^5
\]
for some $\gamma_1,\sigma_1$. As before, we may assume $\nu = 0$ by completing the square. In this way, the only case that remains is
\[
f=x^2+y^3 +\eta_1yz^4+\eta_2 z^6.
\]
Nevertheless, we claim that in that case $R=S/f$ is not a rational singularity. We will show that the blowup along the closed point is not rational, but if $R$ is a rational singularity then such a blowup have to be rational by \cite[
Proposition~8.1]{LipmanRationalSingularities}. Precisely, let $X$ be the blowup of $R$ along $\fram = \fran/f$. Then one chart of $X$ is given by the spectrum of $S[x/z, y/z]/(f/z^2)$, where by $f/z^2$ we mean
\[
f/z^2=(x/z)^2 + (y/z)^3 z  + \eta_1(y/z) z^3+\eta_2 z^4
\]
Setting $\tilde x=(x/z)$ and $\tilde y=(y/z)$, the above equation is $\tilde x^2+\tilde y^3z+\eta_1\tilde yz^3+\eta_2z^4$ so it is not a rational singularity as in the proof of \autoref{lem:Three possible forms} (note that this equation has the form $\tilde x^2+g$ where $g$ has order $4$).
\end{proof}

\section{Finite covers of rational double points by regular schemes}
\label{sec.FiniteCoversOfRationalDoublePoints}
In this section we prove Theorem~\autoref{Main Theorem covers by regular ring}, we first notice that any $2$-dimensional Gorenstein rational singularity is a rational double point (in particular, it is a hypersurface of multiplicity $2$). This fact is well-known but we cannot find a good reference in mixed characteristic. Thus we include a short argument.
\begin{lemma}
\label{lem.BS}
Let $(R,\fram,\kay)$ be a Gorenstein rational singularity of dimension $2$ that is not regular. Then $R$ has multiplicity $2$ and embedding dimension $3$. In particular, $\widehat{R}\cong S/(f)$ where $(S,\n,\kay)$ is a regular local ring of dimension $3$ and $f\in \n^2-\n^3$.
\end{lemma}
\begin{proof}
We may assume $\kay$ is an infinite field. Let $(x,y)$ be a minimal reduction of $\fram$. By \cite[Theorem 2.1]{LipmanTeissierPseudoRational}, $\fram^2\subset (x,y)$. Thus, $\fram/(x,y)\subset 0:_{R/(x,y)}\fram$. Now we have
$$e(R)=l(R/(x,y))=1+l(\fram/(x,y))\leq 1+l(0:_{R/(x,y)}\fram)=2,$$
where the last equality is because $R/(x,y)$ is an Artinian Gorenstein ring so it has a $1$-dimensional socle. Thus, $e(R)=2$. But then we know that the embedding dimension of $R$ is less than or equal to $e(R)+1=3$. Therefore, $\widehat{R}$ is a hypersurface.
\end{proof}

For the rest of the proof of Theorem~\autoref{Main Theorem covers by regular ring}, we aim to show that if $X$ is a strictly Henselian rational double point of mixed characteristic $(0,p>5)$ then there exists a split finite cover of $X$ by a regular scheme. Throughout the rest of this section, $(S,\fran,\kay)$ is a regular local ring and we are in the end concerned with the case where it is of mixed characteristic $(0,p>5)$ and $(R,\fram,\kay)=S/(f)$ is the local ring of a rational double point. The proof of Theorem~\autoref{Main Theorem covers by regular ring} is separated into the cases defined by the type of the singularity of $R$ and is organized as follows:

\begin{itemize}
\item If $X$ is a rational double point of type $\textnormal{A}_n$ then there exists a cyclic cover of $X$ by a regular scheme by \autoref{thm.CoversAn};
\item If $X$ is a rational double point of type $\textnormal{D}_n$ then there exists a cyclic cover of $X$ by a rational double point of type $\textnormal{A}_{2n-5}$ by \autoref{thm.CoversDn} and   \autoref{thm.CoversDnPartII};
\item If $X$ is a rational double point of type $\textnormal{E}_6$ then there is a cyclic cover of $X$ by a rational double point of type $\textnormal{D}_4$ by \autoref{thm.CoversE6};
\item If $X$ is a rational double point of type $\textnormal{E}_7$ then there is a cyclic cover of $X$ by a rational double point of type $\textnormal{E}_6$ by \autoref{thm.CoversE7} and  \autoref{cor: E7 Henselian};
\item If $X$ is a rational double point of type $\textnormal{E}_8$ then an explicit description of a finite split cover of $X$ by a regular scheme is provided in \autoref{thm.CoversE8}.\footnote{Rational double points of type $\textnormal{E}_8$ are seen to be unique factorization domains, \cite[Theorem~25.1]{LipmanRationalSingularities}. Therefore the only cyclic cover of an $\textnormal{E}_8$ singularity is the identity map and therefore an alternative approach to finding a split finite cover of $\textnormal{E}_8$ singularities by a regular scheme is necessary.}
\end{itemize}

\begin{proposition}[Type $\textrm{A}_n$] \label{thm.CoversAn}
Suppose that $S$ is strictly Henselian of residual characteristic $p > 2$, let $R$ be of type $\mathrm{A}_n$, and write $R\cong S/(x^2+y^2+z^{n+1})$ where $x,y,z$ is a choice of minimal generators of the maximal ideal of $S$; see \autoref{prop.ThreeSquares} and \autoref{prop.AnClassification}. Then, $\mathfrak{p}=(x+iy,z)$ is a pure height-$1$ prime ideal of $R$ of order $n+1$ as an element of the divisor class group, $\mathfrak{p}^{(n+1)}= (x+iy)$, and the corresponding index $n+1$ cyclic cover is regular.
\end{proposition}
\begin{proof}
For sake of notation, we write $x_0=x-iy$ and $y_0=x+iy$. We will first show that the divisor class of $\mathfrak{p}=(y_0,z)$ in $\Cl (R)$ has order $n+1$, and $\mathfrak{p}^{(k)} = \bigl(y_0,z^k\bigr)$.
\begin{claim}
\label{clm.kthSymbolicAn}
$\mathfrak{p}^{(k)} = \bigl(y_0,z^k\bigr)$ for all $k\leq n$, and $\mathfrak{p}^{(n+1)}=(y_0)$.
\end{claim}
\begin{proof}
First notice that $\mathfrak{p}^kR_{\mathfrak{p}}=z^k R_{\mathfrak{p}} = \bigl(y_0,z^k\bigr)R_{\mathfrak{p}}$ for all $0\leq k \leq n+1$, since $x_0y_0 = -z^{n+1-k} z^k$ and $x_0\notin \mathfrak{p}$. Therefore, by \cite[Proposition 4.8 (ii)]{AtiyahMacdonald}, it suffices to prove that $(y_0,z^k)$ is a $\mathfrak{p}$-primary ideal of $R$. However, this is clear as every zerodivisor of the ring
\[
R\big/\bigl(y_0,z^k\bigr) \cong S\big/\bigl(y_0,z^k\bigr)
\]
is nilpotent, and the ring has depth $1$ as it is a complete intersection.
In particular, $\fram = (x_0,y_0,z)$ cannot be an associated prime of $\bigl(y_0,z^k\bigr)$. Thus,  $(y_0,z^k)$ is $\mathfrak{p}$-primary.
\end{proof}

We construct the cyclic cover associated to the isomorphism $\mathfrak{p}^{(n+1)} = (y_0) \cong R$ where $y_0 \mapsto 1$.  We write
\[
C = \bigoplus_{k=0}^{n} \bigl(y_k,z_1^k\bigr)_R
\]
where $y_k$ denotes the copy of $y_0$ in the degree-$k$ direct summand of the cyclic cover $C$ corresponding to $\mathfrak{p}$, and $z_1$ denotes the copy of $z$ in the degree-$1$ direct summand of $C$. 

We have that $C$ is a domain by \autoref{prop.Cyclic covers are domains}, and is local with maximal ideal $\mathfrak{c} = \fram \oplus \bigoplus_{k=1}^{n} \mathfrak{p}^{(k)}$, see \cite[Proposition 4.21]{CarvajalFiniteTorsors}. Therefore, in order to show $C$ is regular it suffices to prove that $\mathfrak{c}$ is generated by two elements. To this end, notice that we have $\mathfrak{c}=(x_0,y_0,z,z_1, y_1,\ldots,y_n)$. However, these elements are subject to the following relations:
\begin{enumerate}[(1)]
\item $z_1^{n+1}=-x_0$,
\item $y_1^{n+1}=y_0^n$,
\item $y_1^k = y_0^{k-1} y_k$,
\item $y_1 y_n = y_0$, and
\item $z_1 y_n = z$,
\end{enumerate}
where $k=1,\ldots,n$. These relations imply that $\mathfrak{c}=(y_0, z_1, y_1,\ldots,y_n)$. If we raise the fourth relation $y_1 y_n = y_0$ to the $n+1-k$ power (with $0 \leq k \leq n-1$), we obtain
\[
y_1^{n+1-k} y_n^{n+1-k} =y_0^{n+1-k}.
\]
By multiplying this equation by $y_1^k$ (in case $k \neq 0$) and using the second relation, we see that
\[
y_0^n y_n^{n+1-k} = y_0^{n+1-k} y_1^{k} .
\]
The third relation, however, implies that the right hand side is equal to $y_0^{n}y_k$ (even if $k=0$). Then, by canceling $y_0^n$ out on both sides (as $C$ is a domain), we obtain the new relation
\[
y_n^{n+1-k} = y_k.
\]
for all $k=0,\ldots,n$. Therefore, $\mathfrak{c}=(z_1,y_n)$ is indeed generated by $2$ elements.
\end{proof}

\begin{proposition}[Type $\textrm{D}_n$: complete case] \label{thm.CoversDn}
Assume that $S$ is complete with separably closed residue field of characteristic $p > 3$, let $R$ be of type $\mathrm{D}_n$, and write $R\cong S/(x^2+y^2z+z^{n-1})$ where $x,y,z$ is a choice of minimal generators of the maximal ideal of $S$; see \autoref{lem:Three possible forms}, \autoref{prop:D4 classification}, and \autoref{prop:Dn classification}. Then $\mathfrak{p}=(x,z)$ is a pure height $1$ prime ideal of $R$ of order $2$ as an element of the divisor class group and the corresponding cyclic cover is a singularity of type $\mathrm{A}_{2n-5}$.
\end{proposition}
\begin{proof}
First, we prove that the divisor corresponding to $\mathfrak{p}$ has index $2$.
\begin{claim}
\label{clm.SecondSymbolicDn}
$\mathfrak{p}^{(2)}=(z)$.
\end{claim}
\begin{proof}
First of all, notice that
\[
\mathfrak{p}^2R_{\mathfrak{p}}= \bigl(x^2,xz,z^2\bigr) R_{\mathfrak{p}} = z R_{\mathfrak{p}}
\]
as $x^2=-y^2z-z^{n-1}$. In this way, it suffices to observe that $(z)$ is a $\mathfrak{p}$-primary ideal of $R$ \cite[Proposition 4.8 (ii)]{AtiyahMacdonald}.
\end{proof}
Let $C=R \oplus ( x_1,z_1)_R$ be the cyclic cover corresponding to $\mathfrak{p}$. Then, we have the following relations:
\[
x_1^2=-y^2-z^{n-2}; \quad x_1 z_1 = x; \quad z_1^2 = z.
\]
This implies that the maximal ideal of $C$ is $\mathfrak{c}=(x,y,z,x_1,z_1)=(y,x_1,z_1)$. Therefore, there is an isomorphism of $R$-algebras
\[
C' \coloneqq R[\xi,\zeta]\big/\bigl(\xi^2+y^2+z^{n-2}, \xi \zeta - x, \zeta^2 - z\bigr) \to C, \, \xi \mapsto x_1, \, \zeta \mapsto z_1.
\]
Observe that $C'$ is local with maximal ideal $\mathfrak{c'}=(y, \xi, \zeta)$ and residue field $\kay$. On the other hand, note that
\[
C' \cong S[\xi, \zeta]\big/ \bigl(x^2+y^2z+z^{n-1}, \xi^2+y^2+z^{n-2}, \xi \zeta - x, \zeta^2 - z\bigr).
\]
However, the ideal we are modding out by equals
\[
\bigl( \xi^2 + y^2 + \zeta^{2n-4}, \xi \zeta-x, \zeta^2-z  \bigr)
\]
as
\[
x^2 + y^2 z +z^{n-1} \equiv (\xi \zeta)^2+y^2 \zeta^2 +\bigl(\zeta^2\bigr)^{n-1} \equiv \zeta^2\bigl(\xi^2+y^2+\zeta^{2n-4}\bigr) \bmod{\bigl(\xi\zeta - x, \zeta^2-z\bigr)},
\]
and similarly
\[
\xi^2+y^2+z^{n-2} \equiv \xi^2 +y^2 + \zeta^{2n-4} \bmod {\bigl(\xi\zeta - x, \zeta^2-z\bigr)}.
\]
Therefore,
\[
C'\cong S[\xi,\zeta]\big/\bigl( \xi^2 + y^2 + \zeta^{2n-4}, \xi \zeta-x, \zeta^2-z  \bigr) \cong \Bigl( S[\xi,\zeta]\big/\bigl(\xi \zeta-x, \zeta^2-z \bigr) \Bigr) \Big/ \Bigl( \xi^2 + y^2 + \zeta^{2n-4} \Bigr)
\]
where we see that $S[\xi,\zeta]\big/\bigl(\xi \zeta-x, \zeta^2-z \bigr)$ is a $3$-dimensional complete regular local ring with maximal ideal $(\xi, y, \zeta)$. Thus, we have shown that $C$ is isomorphic to a ring of type $\textnormal{A}_{2n-5}$.
\end{proof}

\begin{corollary}[Type $\textrm{D}_n$: Henselian case]\label{thm.CoversDnPartII}
Assume that $S$ is strictly Henselian of residual characteristic $p > 3$ and suppose that $R$ is of type $\textnormal{D}_n$. Then there exists a finite split cover of $\Spec(R)$ by a regular scheme.
\end{corollary}
\begin{proof}
The divisor class group of a strictly Henselian $2$-dimensional local ring with rational singularity is unchanged under completion by results of \cite{LipmanRationalSingularities}. Indeed, by \cite[Proposition~17.1]{LipmanRationalSingularities}, the divisor class group of $R$ is the same as the group $H$ defined on \cite[Page 222 (3)]{LipmanRationalSingularities}. But by \cite[Proposition 16.3 and Correction on page 279]{LipmanRationalSingularities}, the group $H$ is unaffected when passing from $R$ to $\widehat{R}$. Therefore the divisor class group of $R$ agrees with the divisor class group of $\widehat{R}$. By \autoref{thm.CoversDn}  there exists a height $1$ prime ideal $\mathfrak{p}\subset R$ of order $2$ as an element of the divisor class group corresponding to the $\mathfrak{p}$ in \autoref{thm.CoversDn}.  Since the characteristic does not divide 2, the cyclic cover is unique \'etale locally, see \cite[Definition 2.49]{KollarKovacsSingularitiesBook}, and thus the corresponding cyclic cover is a singularity of type $\textnormal{A}_{2n-5}$.
\end{proof}

\begin{proposition}[Type $\textrm{E}_6$] \label{thm.CoversE6}
Assume that $S$ is strictly Henselian of residual characteristic $p > 5$, let $R$ be of type $\textnormal{E}_6$, and write $R\cong S/(x^2+y^3+z^4)$ where $x,y,z$ is a choice of minimal generators of the maximal ideal of $S$, see \autoref{prop:E6,E7,E8}. Then $\mathfrak{p}=\bigl(y,x+iz^2\bigr)$ is a height $1$ prime ideal of $R$ of order $3$ as an element of the divisor class group and the corresponding cyclic cover is a singularity of type $\mathrm{D}_4$.
\end{proposition}
\begin{proof}
For notation ease, we write $x_0=x+iz^2$. Thus, the defining equation $f=x^2+y^3+z^4$ becomes
\[
f=\bigl(x+iz^2\bigr)(x-iz^2)+y^3= x_0\bigl(x_0-2iz^2\bigr)+y^3=x_0\bigl(x_0+z_0^2\bigr)+y^3,
\]
where $z_0 = (-2i)^{1/2}z$. Note that $x_0,y,z_0$ is a choice of generators of the maximal ideal of $S$.

As before, we start off with the following claim proving the order of $\mathfrak{p}=(x_0,y)$ is $3$.
\begin{claim}
$\mathfrak{p}^{(2)}=\bigl(x_0,y^2\bigr)$ and $\mathfrak{p}^{(3)}=\bigl(x_0\bigr)$.
\end{claim}
\begin{proof}
Let $k \leq 3$, we prove $\mathfrak{p}^{k}=\bigl(x_0,y^k\bigr)$. To this end, we notice that these two ideals coincide after we localize at $\mathfrak{p}$, for $x_0\bigl(x_0+z_0^2\bigr)+y^3=0$ and $x_0 + z_0^2 \notin \mathfrak{p}$. Hence, it is enough to prove that $\bigl(x_0,y^k\bigr)$ is a $\mathfrak{p}$-primary ideal of $R$. This follows, as in \autoref{clm.kthSymbolicAn} above, from observing that
\[
R\big/\bigl(x_0, y^k \bigr) \cong S\big/\bigl(x_0, y^k \bigr),
\]
so that all zerodivisors of are nilpotent, and the depth of this ring is $1$, whereby $\fram$ is not primary to $\bigl(x_0,y^k\bigr)$.
\end{proof}
Next, we consider the corresponding Veronese-type cyclic cover
\[
C= R \oplus (x_1, y_1)_R \oplus \bigl(x_2 , y_1^2\bigr)_R
\]
where $x_k$ denotes the copy of $x_0$ in the degree-$k$ direct summand of $C=\bigoplus_{k=0}^2{\mathfrak{p}^{(k)}}$, and similarly, $y_1$ denotes the copy of $y$ in the first direct summand. These variables are subject to the following relations:
\begin{eqnarray}
\label{eqn.relations1} x_1^3 = x_0^2, \\
\label{eqn.relations2} y_1^3+x_0+z_0^2=0,\\
\label{eqn.relations3} x_1 x_2 = x_0, \\
\label{eqn.relations4} x_1 y_1^2 = y^2, \\
\label{eqn.relations5} y_1 x_2 = y.
\end{eqnarray}
Additionally, we deduce that
\[
x_1^2 x_2^2=(x_1x_2)^2 = x_0^2 = x_1^3,
\]
by using the third and first relations. By canceling $x_1^2$ out on both sides (as $C$ is a domain), we obtain the extra relation
\begin{equation} \label{eqn.relation6}
    x_2^2 = x_1.
\end{equation}
Then, we see that the fourth relations follows from this equation and the fifth relation. \emph{A priori}, the maximal ideal of $C$ is $\mathfrak{c}=(x_0,y,z_0,x_1,y_1,x_2)$, but given the above constraints, we see that $\mathfrak{c}=(z_0,y_1,x_2)$. Also, we conclude that
\begin{equation} \label{eqn.relation7}
x_2^3 = x_2^2 x_2 = x_1 x_2 = x_0
\end{equation}
by using \autoref{eqn.relation6} and \autoref{eqn.relations3}. In summary, we can see that we have an isomorphism of $R$-algebras
\[
C' \coloneqq R[\gamma, \xi]\big/\bigl(\gamma^3+x_0+z_0^2, \xi^3 - x_0, \gamma \xi -y \bigr) \to C, \, \gamma \mapsto y_1, \, \xi \mapsto x_2.
\]
On the other hand, we note that $C'$ is isomorphic to $S[\gamma, \xi]$ modulo the ideal
\[
\bigl(y^3+x_0^2+x_0z_0^2, \gamma^3+x_0+z_0^2, \xi^3 - x_0, \gamma \xi -y \bigr).
\]
However, modulo $\bigl(\xi^3-x_0,\gamma \xi-y\bigr)$, we have
\[
y^3+x_0^2+x_0z_0^2 \equiv (\gamma \xi)^3 + \bigl(\xi^3\bigr)^2 + \xi^3 z_0^2 = \xi^3 \bigl( \gamma^3 + \xi^3 + z_0^2 \bigr) \textnormal{ and } \gamma^3+x_0+z_0^2 \equiv \gamma^3 + \xi^3 + z_0^2.
\]
Hence, the above ideal equals the ideal
\[
\bigl(\gamma^3 + \xi^3 + z_0^2, \xi^3 - x_0, \gamma \xi -y
\bigr).
\]
In this manner,
\[
C' \cong \Bigl( S[\gamma,\xi]\big/\bigl(\xi^3 - x_0, \gamma \xi -y\bigr) \Bigr)\Big/\Bigl(\gamma^3 + \xi^3 + z_0^2\Bigr)
\]
where $S[\gamma,\xi]\big/\bigl(\xi^3 - x_0, \gamma \xi -y\bigr)$ is a $3$-dimensional regular local ring with maximal ideal $(z_0,\gamma,\xi)$. To see this last quotient is a $\mathrm{D}_4$ singularity, we may pass to the completion and notice that the sum of cubes can be factored as the product of three different lines; see \autoref{lem:Three possible forms} and  \autoref{prop:D4 classification}.
\end{proof}

\begin{proposition}[Type $\textrm{E}_7$: complete case]
\label{thm.CoversE7}
Assume that $S$ is complete with separably closed residue field of characteristic $p > 5$, let $R$ be of type $\textnormal{E}_7$, and write $R\cong S/(x^2+y^3+yz^3)$ where $x,y,z$ is a choice of minimal generators of the maximal ideal of $S$, see \autoref{prop:E6,E7,E8}. Then $\mathfrak{p}=\bigl(x,y\bigr)$ is a height $1$ prime ideal of $R$ of order 2 in the divisor class group and the corresponding cyclic cover is a singularity of type $\mathrm{E}_6$.
\end{proposition}
\begin{proof}
From the relation $x^2+y^3+yz^3=0$, we see that $\mathfrak{p}^{(2)}=(y)$, so the assertion about the torsion of the divisor class of $\mathfrak{p}$ follows. Consider the associated cyclic cover
\[
C=R \oplus (x_1, y_1)_R
\]
with relations $x_1^2+y^2+z^3=0$, $x_1 y_1 = x$, and $y_1^2=y$. In particular, the maximal ideal of $C$ is $\mathfrak{c}=(x_1,y_1,z)$.

Next, we consider the isomorphisms:
\begin{align*}
C &\cong R[\xi, \gamma]\big/\bigl( \xi^2 +y^2 +z^3, \xi \gamma - x, \gamma^2 - y \bigr)\\
&\cong S[\xi, \gamma] \big/ \bigl( x^2+y^3+yz^3, \xi^2 +y^2 +z^3, \xi \gamma - x, \gamma^2 - y   \bigr) \\
&= S[\xi, \gamma] \big/ \bigl(\xi^2 +\gamma^4 +z^3, \xi \gamma - x, \gamma^2 - y   \bigr)  \\
&\cong \Bigl( S[\xi, \gamma] \big/ \bigl( \xi \gamma - x, \gamma^2 - y \bigr) \Bigr) \Big/ \Bigl( \xi^2 +\gamma^4 +z^3 \Bigr)
\end{align*}
where $S[\xi, \gamma] \big/ \bigl( \xi \gamma - x, \gamma^2 - y \bigr)$ is a $3$-dimensional regular complete local ring with maximal ideal $(\xi,\gamma,z)$. Hence, $C$ is of type $\mathrm{E}_6$.
\end{proof}

\begin{corollary}[Type $\textrm{E}_7$: Henselian case]
\label{cor: E7 Henselian} Assume that $S$ is strictly Henselian with residual characteristic $p > 5$ and suppose that $R$ is of type $\textnormal{E}_7$. Then there exists a finite split cover of $\Spec(R)$ by a regular scheme.
\end{corollary}

\begin{proof}
The proof technique is identical to that of \autoref{thm.CoversDnPartII}.
\end{proof}

\begin{proposition}[Type $\textrm{E}_8$: complete case]
 \label{thm.CoversE8}
 Assume that $S$ is complete with separably closed residue field of characteristic $p > 5$, let $R$ be of type $\textnormal{E}_8$, and write $R\cong S/(x^2+y^3+z^5)$ where $x,y,z$ is a choice of minimal generators of the maximal ideal of $S$, see \autoref{prop:E6,E7,E8}. Then there exists a split finite cover of $\Spec(R)$ by a regular scheme.  Furthermore, this cover is \'etale on the punctured spectrum.
\end{proposition}
\begin{proof}
Let $f_1$, $f_2$, and $f_3$ be the following polynomials in $\mathbb{Q}(\sqrt[3]{2},\sqrt[3]{3})[u,v]$:
\begin{align*}
    f_1 &= u^{30}+522u^{25}v^5-10005u^{20}v^{10}-10005u^{10}v^{20}-522u^5v^{25}+v^{30};\\
    f_2 &= -\sqrt[3]{2^{8}\cdot 3^4}(u^{20}-228u^{15}v^5+494u^{10}v^{10}+228u^5v^{15}+v^{20});\\
    f_3 &= -u^{11}v-11u^6v^6+uv^{11}.
\end{align*}

By direct computation, one checks these polynomials satisfy the equation
\begin{equation} \label{eqn.fiEquation}
f_1^2+f_2^3+f_3^5 = 0.
\end{equation}
Note we initially found these $f_i$ by using the {\tt InvariantRing} package \cite{InvariantRingSource} for Macaulay2 \cite{M2}, see also \cite{HawesComputingTheInvariantRing}.  Explicitly, we considered the action of the icosahedral group on a polynomial ring over a field.
J.~Lipman pointed out to us that the same expressions appear also in \cite[Chapter 13]{KleinLecturesIcosahedron}.  These $f_i$ also work in our more general setting, as we demonstrate below.

Now, we fix an isomorphism
\[
S \cong W(\kay)\llbracket x,y,z \rrbracket\big/\bigl(p-Q(x,y,z)\bigr)
\]
where $p$ is the residual characteristic of $S$ and $Q(x,y,z) \in  W(\kay)\llbracket x,y,z \rrbracket$.
\footnote{We are not assuming $Q \in \fran^2$, \eg if $Q = z$, then $S = W(\kay)\llbracket x,y \rrbracket$ corresponds to the unramified case.} In this fashion, our singularity $R$ can be assumed to be
\[
R = W(\kay)\llbracket x,y,z \rrbracket\big/\bigl(p-Q(x,y,z), x^2+y^3+z^5\bigr).
\]
Next, we let $A$ be the (possibly ramified) $2$-dimensional complete regular local $W(\kay)$-algebra
\[
A\coloneqq W(\kay)\llbracket u,v\rrbracket\big/ \bigl( p-Q(f_1, f_2, f_3)\bigr),
\]
here we view $f_1, f_2, f_3 \in W(\kay)\llbracket u,v\rrbracket$ in the obvious way.  Consider the map of $W(\kay)$-algebras $R \to A$ given by sending $x, y, z$ to $f_1, f_2, f_3$, respectively. This map is well-defined because the $f_i$ satisfy the equation \autoref{eqn.fiEquation}.

It remains to prove that $R \to A$ is a finite split extension. To this end, it suffices to prove it is a finite extension of degree $120$, an invertible element in $W(\kay)$, for in that case the trace map $\Tr_{A/R} \: A \to R$ can be used to split the extension $R \to A$.

In order to prove that $R \to A$ is a finite extension of degree $120$, we notice that the map of $W(\kay)$-algebras
\[
W(\kay)\llbracket x,y,z \rrbracket\big/\bigl( x^2+y^3+z^5\bigr) \to W(\kay)\llbracket u,v \rrbracket
\]
obtained by sending $x,y,z$ to $f_1, f_2, f_3$, respectively, is a finite extension of degree $120$. Indeed, if we invert $p$, this follows from the equicharacteristic zero case. Hence, the result follows from the following general fact applied to $\mathfrak{p} = (p-Q)$.
\begin{claim}
Let $A \to B$ be a finite extension of domains of (generic) degree $d$ such that $B$ is Cohen--Macaulay. For any prime ideal $\mathfrak{p}$ of $A$ such that $\pd(A/\mathfrak{p})$ is finite, we have that $A/\mathfrak{p} \to B/\mathfrak{p}B$ is a finite extension of (generic) degree $d$.
\end{claim}
\begin{proof}
Since $\pd(A/\mathfrak{p})$ is finite, we have that $\textnormal{A}_{\mathfrak{p}}$ is a regular local ring. Since $B$ is Cohen--Macaulay, so is $B_{\mathfrak{p}}$ and thus the extension $A_{\mathfrak{p}} \to B_{\mathfrak{p}}$ is finite free, whose rank must be $d$. Of course, the rank $A_{\mathfrak{p}} \to B_{\mathfrak{p}}$ is nothing but its residual degree, which is the (generic) degree of $A/\mathfrak{p} \to B/\mathfrak{p}B$.
\end{proof}

Finally, we prove that $R \to A$ is \'etale on the punctured spectrum.  It suffices to show that the map $W(\kay)\llbracket x,y,z \rrbracket\big/\bigl( x^2+y^3+z^5\bigr) \to W(\kay)\llbracket u,v \rrbracket$ is \'etale on the punctured spectrum.  But this can be checked by hand (or in a computer algebra system working over $\mathbb{Z}$ with $2,3,5$ inverted).  This completes the proof.
\end{proof}

\begin{corollary}[Type $\textrm{E}_8$: Henselian case]
\label{cor: E8 Henselian} Assume that $S$ is strictly Henselian of residual characteristic $p > 5$ and suppose that $R$ is of type $\mathrm{E}_8$. Then there exists a finite split cover of $\Spec(R)$ by a regular scheme.
\end{corollary}
\begin{proof}
Using \cite[Page 579]{ElkikSolutionsdEquations}, \cf \cite[Page 3]{BhattGabberOlssonSpreadOut}, the local fundamental group (the fundamental group of the punctured spectrum) of $R$, is unchanged by passage to completion.  Hence, there exists a finite extension $R \to A'$, \'etale on the punctured spectrum, such that $\widehat{R} \to \widehat{A'}$ is isomorphic to the extension constructed in \autoref{thm.CoversE8}.  Since the completion of $A'$ is regular, so is $A'$, and the result follows.
\end{proof}

We end this section by giving an example showing that, similar to the equicharacteristic $p>0$ scenario, rational double points of mixed characteristic $(0,p)$ with $p$ small are not always direct summands of regular rings.

\begin{example}
\label{example}
Let $R=W(\kay)\llbracket y,z\rrbracket/(p^2+y^3+z^5)$ where $\kay$ is an algebraically closed field of characteristic $p=3$. One can check that $R$ has a rational singularity (for example, use \cite[Proposition 8.1]{LipmanRationalSingularities}). We claim that $R$ cannot be a direct summand of a regular ring. In fact, direct summands of regular rings are splinters (in mixed characteristic this follows from \cite{AndreDSC}). Thus if $R$ is a direct summand of a regular ring, then for any module-finite extension $T$ of $R$, $y^2\notin (p,z)T$ (since $y^2\notin (p,z)R$). However, this is not true by the following claim.
\begin{claim}
There exists a module-finite extension $T$ of $R$ such that $y^2\in (p,z)T$.
\end{claim}
\begin{proof}
Recall that $p=3$, so we write $y^2=3v+zu$ and we will solve $u$ and $v$ using monic equations over $R$. Since $v=(y^2-uz)/3$, we have
$$v^3=\frac{y^6-u^3z^3+3y^2u^2z^2-3y^4uz}{27} \hspace{1em} \text{ and } \hspace{1em} y^2v^2=\frac{3y^6-6y^4uz+3y^2u^2z^2}{27}.$$
Thus we have
$$v^3-y^2v^2=\frac{-2y^6-u^3y^3+3y^4uz}{27}.$$
Plug in $y^3=-9-z^5$ and $y^6=81+z^{10}+18z^5$ to the above equation to have
$$v^3-y^2v^2+(6+yuz)+\frac{z^3(u^3+3yz^3u+2z^7+36z^2)}{27}=0.$$
Consider the following equations on $u,v$:
\[\begin{cases}
  u^3+3yz^3u+(2z^7+36z^2)=0  \\
  v^3-y^2v^2+(6+yuz)=0.
\end{cases}\]
Since both equations are monic, it is clear that there exists a module-finite extension $T$ of $R$ such that $u,v$ has solutions in $T$. Thus working backwards we see that $y^2\in(p,z)T$.
\end{proof}
\end{example}

\section{Cyclic covers of \BCMReg{} singularities}
\label{sec.CyclicCoversOfBCM}

In this final section, we give an application of Theorem~\autoref{Main Theorem covers by regular ring} to \BCMReg{B} singularities introduced in \cite{MaSchwedeSingularitiesMixedChar}. We first collect some notations from \cite{MaSchwedeSingularitiesMixedChar}.
Let $(R,\fram)$ be a complete local ring of dimension $d$.  For every big Cohen--Macaulay $R$-algebra $B$ we define
\[
0^B_{H_\m^d(R)}=\ker\big(H_\m^d(R)\to H_\m^d(B)\big).
\]
We further define
\[
\mytau_B(\omega_R)=\Ann_{\omega_R}0^B_{H_\m^d(R)}=\Big(H_\m^d(R)\Big/0^B_{H_\m^d(R)} \Big)^\vee \subset \omega_R
\]
to be the BCM parameter test submodule. If $(R,\fram)$ is a complete normal local domain, $\Delta \geq 0$ is an effective $\bQ$-divisor such that $K_R+\Delta$ is $\bQ$-Cartier, and $B$ is a big Cohen--Macaulay $R^+$-algebra,\footnote{Here $R^+$ denote the absolute integral closure of $R$: the integral closure of $R$ inside an algebraic closure of its fraction field.} then we can define the BCM test ideal $\mytau_B(R,\Delta)$ as a variant of $\mytau_B(\omega_R)$ (see \cite[Definition 6.9]{MaSchwedeSingularitiesMixedChar} for the detailed definition). A $\bQ$-Gorenstein complete normal local domain $(R,\fram)$ is called \BCMReg{B} if $\mytau_B(R)\coloneqq \mytau_B(R,0)=R$. For the purpose of this paper, we point out that $(R,\fram)$ is \BCMReg{B} if and only if $R\to B$ is pure \cite[Theorem 6.12, Proposition 6.14]{MaSchwedeSingularitiesMixedChar}.

\begin{lemma}
\label{lem.CyclicCoverOfBCMRegIsBCMReg}
Suppose that $(R, \fram)$ is a normal $\bQ$-Gorenstein local ring of mixed characteristic $(0, p)$.  Suppose that $D \geq 0$ is a Weil divisor of index $n$ on $\Spec R$.  Let $S = \bigoplus_{i = 0}^{n-1} R(iD)$ denote the cyclic cover.  Let $B$ be a big Cohen--Macaulay $\widehat{R}^+$-algebra. Then $\widehat{R}$ is \BCMReg{B} if and only if $\widehat{S}$ is \BCMReg{B}.
\end{lemma}
\begin{proof}
Without loss of generality, we may assume that $R$ is complete.
If $S$ is \BCMReg{B}, it is a normal domain, then $R$ is \BCMReg{B} since $R \to S$ splits, this follows from \cite[Theorem 6.12, Proposition 6.14]{MaSchwedeSingularitiesMixedChar}.

Conversely, suppose that $R$ is \BCMReg{B}.  In the case that $p$ does not divide $n$, $S$ is normal and \cite[Corollary 6.20]{MaSchwedeSingularitiesMixedChar} implies that $S$ is \BCMReg{B}.

Thus, we assume noe $p \mid n$. Note that $S$ is still a domain by \autoref{prop.Cyclic covers are domains}. Even though $S$ is not necessarily normal, it is G1 (Gorenstein in codimension 1) since in codimension 1 it is a cyclic cover of a Gorenstein ring (localizing at a height one prime, we are adjoining a single variable and modding out by a single equation).
Write the induced map $\pi : \Spec S \to \Spec R$ and notice that we can pull back $\pi^* D$ to obtain an almost Cartier divisor in the sense of \cite{HartshorneGeneralizedDivisorsOnGorensteinSchemes}.  We also have that $\pi^* D$ and $\pi^* K_R$ are $\bQ$-Cartier.  Notice finally that $\pi^*D$ is in fact Cartier since we took a cyclic cover.  Recall that the ring $S$ is local with maximal ideal $\frn := \fram \oplus R(D) \oplus R(2D) \oplus \dots \oplus R((n-1)D)$.  There is also a trace-like map $T : S \to R$ which projects onto the first coordinate, which satisfies $T(\frn) \subset \fram$ and which generates $\Hom_R(S, R)$ as an $S$-module, see for example \cite[Section 4.4]{CarvajalFiniteTorsors}.

Let $\Tr \in \Hom_R(S, R)$ denote the field trace and write $\Tr(-) = T(s \cdot -)$ for some $s \in S$; we define $\Ram_{S/R} := \Div_S(s)$.  Even though $S$ is not necessarily normal, since it is G1 and S2, we may define a canonical divisor $K_S$ with $K_S = \pi^* K_R - \Ram_{S/R} \sim \pi^* K_R$.  It follows that $K_S$ is also $\bQ$-Cartier with $nK_S$ Cartier.  In this setting we may define $\mytau_B(S) := \mytau_B(\omega_S, K_S)$ as in \cite[Definition 6.2]{MaSchwedeSingularitiesMixedChar}, it is still an ideal of $S$ just as in \cite[Lemma 6.8]{MaSchwedeSingularitiesMixedChar}.

Choose a Cartier divisor $H = \Div_R(r)$ on $R$ such that $\pi^* H \geq \Ram_{S/R}$.  We know from \cite[Theorem 6.17]{MaSchwedeSingularitiesMixedChar}, whose proof does not use that $S$ is normal, that
\begin{align*}
r T(\mytau_B(S)) = T(r\mytau_B(S)) = T(r\mytau_B(S, \Ram_{S/R} - \Ram_{S/R})) &= T(s \mytau_B(S, \pi^* H - \Ram_{S/R}))\\
&= \Tr(\mytau_B(S, \pi^* H - \Ram_{S/R})\\
&= \mytau_B(R, H) = r \mytau_B(R).
\end{align*}
Hence $T(\mytau_B(S)) = \mytau_B(R)$. However, since $\mytau_B(R) = R$, and $T(\frn) \subset \fram$, we must have $\mytau_B(S) \not\subset \frn$.  Hence $\mytau_B(S) = S$, which also proves, arguing exactly as in \cite[Theorem 6.12]{MaSchwedeSingularitiesMixedChar} (again working in the G1 and S2 instead of the normal case), that $S \to B$ is pure.  Among other things, since $S \to B$ factors through the normalization of $S$, $S^{\mathrm{N}}$, this implies $S \to S^{\mathrm{N}}$ splits, and hence $S$ is normal. Thus \cite[Proposition 6.14]{MaSchwedeSingularitiesMixedChar} implies that $S$ is \BCMReg{B}.
\end{proof}

\begin{theorem}
Let $(R, \fram, \kay)$ be a strictly Henselian 2-dimensional local ring of mixed characteristic $(0, p>5)$. Let $B$ be a big Cohen--Macaulay $\widehat{R}^+$-algebra. Suppose that $\widehat{R}$ is \BCMReg{B}. Then there exists a finite split extension $R \subset T$ with $T$ regular.
\end{theorem}
\begin{proof}
Since $\widehat{R}$ is \BCMReg{B}, $\widehat{R}$ is \BCMRat{B} and hence pseudo-rational by \cite[Theorem 6.12, Proposition 3.7]{MaSchwedeSingularitiesMixedChar}. Since $R$ has dimension $2$, $R$ is thus a rational singularity and so $R$ is $\bQ$-Gorenstein by \cite[Proposition 17.1]{LipmanRationalSingularities}. Let $S$ be the canonical cover of $R$. Then $\widehat{S}$ is \BCMReg{B} by \autoref{lem.CyclicCoverOfBCMRegIsBCMReg}, so by the same reasoning $S$ is a rational singularity (i.e., by \cite[Theorem 6.12, Proposition 3.7]{MaSchwedeSingularitiesMixedChar}). But now since $S$ is Gorenstein, we can invoke Theorem~\autoref{Main Theorem covers by regular ring} to see that there exists a finite split extension $S\subset T$ such that $T$ is regular. Since $R\to S$ splits, the composition of extensions $R\subset T$ completes the proof.
\end{proof}

\bibliographystyle{skalpha}
\bibliography{MainBib}
\end{document}